\documentclass[12pt]{amsart}
\usepackage{RLintro}

\makeatletter
\def\thickhline{%
  \noalign{\ifnum0=`}\fi\hrule \@height \thickarrayrulewidth \futurelet
   \reserved@a\@xthickhline}
\def\@xthickhline{\ifx\reserved@a\thickhline
               \vskip\doublerulesep
               \vskip-\thickarrayrulewidth
             \fi
      \ifnum0=`{\fi}}
\makeatother

\newlength{\thickarrayrulewidth}
\newcommand\Tstrut{\rule{0pt}{2.6ex}}         
\newcommand\Bstrut{\rule[-0.9ex]{0pt}{0pt}}   

\DeclareMathOperator{\Comm}{Comm}

\DeclareMathOperator{\Isom}{Isom}
\newcommand{\defn}[1]{\textit{#1}}

\newcommand{\QGr}{\mathbf{QGr}}

\newcommand{\Cl}{\mathrm{Cl}}
\newcommand{\gen}{\mathrm{gen}}
\newcommand{\bone}{\mathbf{1}}
\renewcommand{\Vec}{\mathrm{Vec}}
\newcommand{\FI}{\mathbf{FI}}
\newcommand{\SVec}{\mathrm{SVec}}
\newcommand{\prop}[1]{{\rm (#1)}}
\newcommand{\propn}[2]{{\rm (#1\textsubscript{\rm #2})}}

\newcommand{\osp}{\mathfrak{osp}}
\newcommand{\pe}{\mathfrak{pe}}
\newcommand{\alg}{\mathrm{alg}}
\newcommand{\npol}{\mathrm{npol}}
\newcommand{\nalg}{\mathrm{nalg}}
\newcommand{\gl}{\mathfrak{gl}}

\begin{document}

\title[Equivariant primes ideals for infinite dimensional supergroups]{Equivariant primes ideals for \\ infinite dimensional supergroups}
\date{\today}

\author{Robert P. Laudone}
\address{Department of Mathematics, University of Michigan, Ann Arbor, MI}
\email{\href{mailto:laudone@umich.edu}{laudone@umich.edu}}
\urladdr{\url{https://robertplaudone.github.io/}}

\author{Andrew Snowden}
\address{Department of Mathematics, University of Michigan, Ann Arbor, MI}
\email{\href{mailto:asnowden@umich.edu}{asnowden@umich.edu}}
\urladdr{\url{http://www-personal.umich.edu/~asnowden/}}

\thanks{RL was supported by NSF grant DMS-2001992. AS was supported by NSF grant DMS-1453893.}

\maketitle

\begin{abstract}
Let $A$ be a commutative algebra equipped with an action of a group $G$. The so-called $G$-primes of $A$ are the equivariant analogs of prime ideals, and of central importance in equivariant commutative algebra. When $G$ is an infinite dimensional group, these ideals can be very subtle: for instance, distinct $G$-primes can have the same radical. In previous work, the second author showed that if $G=\GL_{\infty}$ and $A$ is a polynomial representation, then these pathologies disappear when $G$ is replaced with the supergroup $\GL_{\infty|\infty}$ and $A$ with a corresponding algebra; this leads to a geometric description of $G$-primes of $A$. In the present paper, we construct an abstract framework around this result, and apply the framework to prove analogous results for other (super)groups. We give some applications to the isomeric determinantal ideals (commonly known as ``queer determinantal ideals").
\end{abstract}

\tableofcontents

\section{Introduction}

\subsection{Background}

A \defn{$\GL$-algebra} is a commutative algebra equipped with an action of the infinite general linear group $\GL$ under which it forms a polynomial representation. Over the last decade, mathematicians have realized that these algebras are well-behaved and widely applicable: for example, modules over the simplest $\GL$-algebra $\Sym(\bQ^{\infty})$ are equivalent (via Schur--Weyl duality) to the $\FI$-modules of Church, Ellenberg, and Farb \cite{fimodule}; Sam and the second author \cite{infrank} used $\GL$-algebras to study the stable representation theory of classical groups; and Draisma \cite{draisma} proved a topological noetherianity result for these algebras which has been applied \cite{draisma-lason-leykin, stillman} to give new proofs of Stillman's conjecture.

A \defn{$\GL$-prime} of a $\GL$-algebra $A$ is a $\GL$-stable ideal $\fp$ of $A$ that such that $\fa \fb \subset \fp$ implies $\fa \subset \fp$ or $\fb \subset \fp$ for $\GL$-stable ideals $\fa$ and $\fb$. These ideals take the place of ordinary prime ideals in the equivariant theory, and are therefore of central importance. Any ordinary prime ideal that is $\GL$-stable is $\GL$-prime, but the converse is not true: for example, if $A$ is the even subalgebra of the exterior algebra on the standard representation then the zero ideal is $\GL$-prime. This example shows $\Spec(A)$ cannot ``see'' the $\GL$-primes of $A$. The second author solved this problem (in characteristic~0) in \cite{tcaprimes}, the antecedent of the present paper: if one regards $A$ as a polynomial functor then one can evaluate $A$ on the super vector space $\bQ^{\infty|\infty}$, and the spectrum of the resulting ring is rich enough to detect all $\GL$-primes. This provides a geometric basis for studying these ideals.

The purpose of this paper is twofold. First, we abstract and partially axiomatize the results from \cite{tcaprimes}. And second, we apply this framework to study equivariant primes for other infinite dimensional Lie (super)algebras.

\subsection{Abstract results}

The results of \cite{tcaprimes} compare the equivariant commutative algebra of a $\GL$-algebra $A$ to the ordinary commutative algebra of $A(\bQ^{\infty|\infty})$. In other words, if $\omega \colon \Rep^{\pol}(\GL) \to \SVec$ denotes the functor $\omega(M)=M(\bQ^{\infty|\infty})$, then these results compare the commutative algebra of $A$ and $\omega(A)$, regarded as algebra objects in the respective categories. This suggests the following general problem:

\begin{problem}
Let $\omega \colon \cC \to \cD$ be a tensor functor and let $A$ be a commutative algebra object of $\cC$. How do commutative algebraic properties of $A$ and $\omega(A)$ compare?
\end{problem}

For example, one may ask more specifically: how do the prime ideals of $A$ and $\omega(A)$ compare? (We note that ``prime ideal'' makes sense for commutative algebras in any tensor category. The notion of $\GL$-prime discussed above is the categorical notion of prime for an algebra in $\Rep^{\pol}(\GL)$.) In \S \ref{ss:mainprops}, we introduce two conditions \prop{A} and \prop{B} on $\omega$ that allow for comparison of certain commutative algebraic properties. In this language, the main results of \cite{tcaprimes} simply state that the functor $\Rep^{\pol}(\GL) \to \SVec$ satisfies \prop{A} and \prop{B}. We prove a number of abstract results about these properties, and give various criteria for establishing them. This streamlines the task of establishing these properties in concrete situations.

\subsection{Lie superalgebras}

We consider four infinite dimensional Lie superalgebras:
\begin{equation} \label{eq:lie}
\fgl_{\infty|\infty}, \quad \osp_{\infty|\infty}, \quad \pe_{\infty}, \quad \fq_{\infty}
\end{equation}
(general linear, orthosymplectic, periplectic, and isomeric\footnote{Commonly known as the ``queer superalgebra"}; see \S \ref{s:lie} for definitions). If $\fg$ is any one of these algebras, then there is a tensor category $\Rep^{\alg}(\fg)$ comprising the algebraic representations of $\fg$ (over a fixed field of characteristic~0). Our main theorem is the following:

\begin{theorem} \label{mainthm}
Let $\fg$ be one of the four Lie superalgebras \eqref{eq:lie}. Then the forgetful functor
\begin{displaymath}
\Rep^{\alg}(\fg) \to \SVec
\end{displaymath}
satisfies properties \prop{A} and \prop{B}.
\end{theorem}

We note that this theorem is new even for $\fgl$: the above theorem treats \emph{algebraic} representations of $\fgl$, while \cite{tcaprimes} applies only to the smaller category of \emph{polynomial} representations. (The above theorem applies equally well to the category of polynomial representations of $\fq_{\infty}$.) The most important consequences of the theorem are spelled out below:

\begin{corollary}
Let $\fg$ be as above and let $A$ be a commutative algebra in $\Rep^{\alg}(\fg)$.
\begin{enumerate}
\item Let $\fa$ and $\fb$ be $\fg$-ideals. Then $\rad_{\fg}(\fa) \subset \rad_{\fg}(\fb)$ if and only if $\rad(\fa) \subset \rad(\fb)$.
\item Let $\fa$ be a $\fg$-ideal. Then $\rad_{\fg}(\fa)$ is $\fg$-prime if and only if $\rad(\fa)$ is prime.
\item The construction $\fp \mapsto \rad(\fp)$ defines a bijection between minimal $\fg$-primes of $A$ and minimal primes of $A$.
\end{enumerate}
Now suppose that $A$ is generated over a noetherian coefficient ring by a finite length $\fg$-subrepresentation. Then we also have:
\begin{enumerate}[resume]
\item The $\fg$-spectrum $\Spec_{\fg}(A)$ of $A$ is a noetherian topological space.
\end{enumerate}
\end{corollary}

Here $\rad_{\fg}(\fa)$ is the sum of all $\fg$-stable ideals $\fc$ such that $\fc^n \subset \fa$ for some $n$; see Definition~\ref{def:rad}. Also, $\Spec_{\fg}(A)$ is the set of all $\fg$-primes of $A$, equipped with the Zariski topology; see Definition~\ref{def:spec}.

In fact, it is possible to prove a more general version of Theorem~\ref{mainthm} with our methods. Let $\fg$ and $\fh$ be finite products of the algebras in \eqref{eq:lie} and let $\fh \to \fg$ be a homomorphism built out of the various standard homomorphisms between these algebras. Then the restriction functor $\Rep^{\alg}(\fg) \to \Rep^{\alg}(\fh)$ satisfies \prop{A} and \prop{B}.

The above results imply similar results for (non-super) Lie algebras. For example, consider the infinite orthogonal Lie algebra $\fo_{\infty}$, and let $\Rep^{\alg}(\fo_{\infty})$ be its category of algebraic representations (on ordinary vector spaces). The forgetful functor $\Rep^{\alg}(\fo_{\infty}) \to \Vec$ does \emph{not} satisfy \prop{A}; this can be seen using the aforementioned exterior algebra example. However, the category $\Rep^{\alg}(\fo_{\infty})$ is (essentially) equivalent to $\Rep^{\alg}(\osp_{\infty|\infty})$. Via this equivalence we obtain a functor $\Rep^{\alg}(\fo_{\infty}) \to \SVec$, and it follows from Theorem~\ref{mainthm} that it satisfies \prop{A} and \prop{B}. Thus the equivariant primes in an $\fo_{\infty}$-algebra $A$ can be understood geometrically after replacing $A$ with the corresponding superalgebra in $\Rep^{\alg}(\osp_{\infty|\infty})$.

\subsection{An application}

Let $V$ and $W$ be infinite dimensional isomeric vector spaces, let $U$ be their half tensor product, and let $A=\Sym(U)$ (see \S \ref{s:isomeric} for definitions). We regard $A$ as an algebra object in the category of polynomial representations of $\fq(V) \times \fq(W)$. The algebra $A$ was studied in \cite{isomeric}, where its ideal lattice was determined and a noetherian result established. We apply our theory to determine the equivariant spectrum of this algebra. We show that the isomeric analog of determinantal ideals are $(\fq(V) \times \fq(W))$-prime, and account for all such prime ideals.

\subsection{Outline}

In \S \ref{s:tencat}, we introduce elementary concepts of commutative algebra in tensor categories, and in \S \ref{s:fiber} we discuss the basic ways in which these concepts interact with tensor functors. In \S \ref{s:axioms}, we formulate the properties \prop{A} and \prop{B} and prove various abstract results about them. In \S \ref{s:lie}, we study commutative algebras equipped with a Lie algebra action in general tensor categories, and give criteria for \prop{A} and \prop{B}. In \S \ref{s:super}, we apply the abstract results to prove our main results on Lie superalgebras. Finally, in \S \ref{s:isomeric}, we carry out our application to the isomeric algebra $A$.

\subsection{Index of Terms}

The following table lists the most important properties defined in the body of the article:
{\small
\begin{table}[htb]
\centering
\begin{tabular}{| lc | lc | lc |}
\hline
 Property & Section Number &Property & Section Number &Property &Section Number \Tstrut \Bstrut \\
\hline
\prop{PI} &\ref{PI} &\prop{B} &\ref{B} &\prop{Gen} &\ref{Gen} \Tstrut \Bstrut \\
\hline
\prop{For} &\ref{For} &\prop{Fin} &\ref{Fin} &\prop{Stab} &\ref{prop: Stab} \Tstrut \Bstrut\\
\hline
\prop{A} &\ref{A} &\prop{Rad} &\ref{prop: Rad} &\prop{UF} &\ref{UF}  \Tstrut \Bstrut \\
\hline
\end{tabular}
\label{table:PropIndex}
\end{table}
}

\section{Commutative algebra in tensor categories} \label{s:tencat}

\subsection{Basic definitions}

In this section, we discuss a few aspects of commutative algebras in tensor categories. We begin by clarifying our notion of tensor category:

\begin{definition} \label{def:tensor-cat}
A \defn{tensor category} is a symmetric monoidal category $(\cC, \otimes)$ such that $\cC$ is a Grothendieck abelian category and $\otimes$ is cocontinuous in each variable. (Thus $\otimes$ is right exact and commutes with all direct sums in each variable.)
\end{definition}

Recall that an object $M$ of an abelian category is \defn{finitely generated} if the following condition holds: given a family $\{N_i\}_{i \in I}$ of subobjects of $M$ such that $M=\sum_{i \in I} N_i$, there exists a finite subset $J$ of $I$ such that $M=\sum_{i \in J} N_i$. This definition coincides with the usual notion of finite generation in most cases, e.g., if $\cC$ is the category of modules over a ring. A general tensor category may not have enough finitely generated objects, and the tensor product may not interact nicely with finite generation. We therefore introduce the following refined notion:

\begin{definition}
A tensor category $\cC$ is \defn{admissible} if it satisfies the following:
\begin{itemize}
\item The unit object $\bone$ is finitely generated.
\item Every object of $\cC$ is the sum of its finitely generated subobjects.
\item The tensor product of two finitely generated objects is finitely generated. \qedhere
\end{itemize}
\end{definition}

Fix an admissible tensor category $\cC$. For an object $M$ of $\cC$, we let $[M]$ or $[M]_{\cC}$ denote the set of all subobjects of $M$. (Note: this is a set since Grothendieck abelian categories are well-powered.) We let $[M]^{\rf}$ or $[M]^{\rf}_{\cC}$ denote the set of all finitely generated subobjects of $M$. We use $[M]$ or $[M]^{\rf}$ as a replacement for the set of elements of $M$. Note that if $K$ and $N$ are subobjects of $M$ then $K \subset N$ holds if and only if $X \in [K]^{\rf}$ implies $X \in [N]^{\rf}$ (admissibility is crucial here since we are taking $X$ finitely generated).

We let $\Comm(\cC)$ be the category of commutative (and associative and unital) algebras in $\cC$. For $A \in \Comm(\cC)$, we let $\Mod_A$ be the category of $A$-modules in $\cC$. An \defn{ideal} of $A$ is an $A$-submodule of $A$. Let $M$ be an $A$-module, let $X \in [A]$, and let $Y \in [M]$. We define $XY \in [M]$ to be the image of the map
\begin{displaymath}
X \otimes Y \to A \otimes M \to M,
\end{displaymath}
where the first map is the tensor product of the inclusions $X \to A$ and $Y \to M$, and the second map is the given map for $M$.

\begin{proposition} \label{prop:multiplication}
Let $A \in \Comm(\cC)$ and let $M$ be an $A$-module.
\begin{enumerate}
\item Let $\{X_i\}_{i \in I}$ be elements of $[A]$ and let $\{Y_j\}_{j \in J}$ be elements of $[M]$. Then we have $(\sum_{i \in I} X_i)(\sum_{j \in J} Y_j)=\sum_{i \in I, j \in J} X_i Y_j$.
\item Let $X,Y \in [A]$ and $Z \in [M]$. Then $XY=YX$ and $(XY)Z=X(YZ)$.
\item If $X \in [A]^{\rf}$ and $Y \in [M]^{\rf}$ then $XY \in [M]^{\rf}$.
\item Suppose $X \subset X'$ belong to $[A]$ and $Y \subset Y'$ belong to $[M]$. Then $XY \subset X'Y'$.
\item Let $N$ be a $\cC$-subobject of $M$. Then $N$ is an $A$-submodule if and only if $AN \subset N$.
\item $M$ is finitely generated as an $A$-module if and only if there exists $X \in [M]^{\rf}$ such that $M=AX$.
\item Let $\bone$ be the unit object of $\cC$, let $\bone_A$ be the image of the natural map $\bone \to A$, and let $Y \in [M]$. Then $\bone_A \cdot Y=Y$.
\end{enumerate}
\end{proposition}

\begin{proof}
We leave this to the reader.
%
\end{proof}

\begin{remark} \label{rmk:fgA}
Since $\cC$ is admissible, $\bone$ is finitely generated and so $\bone_A$ is also finitely generated. Since $A=A \cdot \bone_A$ by (b) and (g), we see from (f) that $A$ is finitely generated as an $A$-module. Without the admissibility condition, this need not be true!
\end{remark}

\subsection{Prime ideals}

Fix an admissible tensor category $\cC$ and $A \in \Comm(\cC)$.

\begin{definition}
Let $\fp$ be an ideal of $A$. We say that $\fp$ is \defn{prime} if $XY \subset \fp$ implies $X \subset \fp$ or $Y \subset \fp$ for all $X,Y \in [A]$. We say that $A$ is \defn{integral} or a \defn{domain} if the zero ideal is prime.
\end{definition}

\begin{proposition}
Let $\fp$ be an ideal of $A$. The following are equivalent:
\begin{enumerate}
\item $\fp$ is prime.
\item $XY \subset \fp$ implies $X \subset \fp$ or $Y \subset \fp$ for all $X,Y \in [A]^{\rf}$.
\item $\fa \fb \subset \fp$ implies $\fa \subset \fp$ or $\fb \subset \fp$ for all ideals $\fa$ and $\fb$ of $A$.
\item $\fa \fb \subset \fp$ implies $\fa \subset \fp$ or $\fb \subset \fp$ for all finitely generated ideals $\fa$ and $\fb$ of $A$.
\end{enumerate}
\end{proposition}

\begin{proof}
We leave this to the reader.
\end{proof}

We say that a prime of $A$ is \defn{minimal} if it does not strictly contain another prime. These always exist:

\begin{proposition}
Let $\fp$ be a prime ideal of $A$. Then there exists a minimal prime ideal $\fq$ of $A$ contained in $\fp$.
\end{proposition}

\begin{proof}
Let $S$ be the set of all prime ideals of $A$ contained in $\fp$. This set is non-empty since it contains $\fp$. Let $\{\fq_i\}_{i \in I}$ be a descending chain in $S$, and put $\fq=\bigcap_{i \in I} \fq_i$. We claim that $\fq$ is prime. Suppose $XY \subset \fq$ for $X,Y \in [A]$; we show that $X \subset \fq$ or $Y \subset \fq$. If $X \subset \fq$ we are done; suppose this is not the case. Then there is some $i \in I$ such that $X_i \not\subset \fq_i$, and so $X \not\subset \fq_j$ for all $j \ge i$. Since $XY \subset \fq_j$, we must have $Y \subset \fq_j$ for all $j \ge i$. Thus $Y \subset \fq$, as claimed. Zorn's lemma now shows that $S$ has a minimal element, which completes the proof.
\end{proof}

\subsection{Radicals}

Let $\cC$ and $A \in \Comm(\cC)$ be as above.

\begin{definition} \label{def:rad}
Let $\fa$ be an ideal of $A$. The \defn{radical} of $\fa$, denoted $\rad(\fa)$, is the sum of all $X \in [A]$ such that $X^n \subset \fa$ for some $n$. The \defn{(nil)radical} of $A$, denoted $\rad(A)$ or $\rad_{\cC}(A)$, is the radical of the zero ideal.
\end{definition}

If $X^n \subset \fa$ then $(AX)^n \subset \fa$ by Proposition~\ref{prop:multiplication}(b). Thus $\rad(\fa)=\sum AX$, where the sum is taken over those $X \in [A]$ with $X^n \subset \fa$ for some $n$. Thus $\rad(\fa)$ is a sum of ideals, and is therefore itself an ideal. If $X \in [\rad(\fa)]$ then we cannot conclude that $X^n \subset \fa$ for some $n$; for example, $\fa$ need not contain a power of $\rad(\fa)$. However, the problem disappears when $X$ is finitely generated:

\begin{proposition} \label{prop:fg-nilp}
Let $\fa$ be an ideal of $A$, and let $X \in [\rad(\fa)]^{\rf}$. Then $X^n \subset \fa$ for some $n$.
\end{proposition}

\begin{proof}
Let $\cU$ be the set of all $Y \in [A]$ such that $Y^n \subset \fa$ for some $n$. Then $\rad(\fa)=\sum_{Y \in \cU} Y$ by definition. Suppose $X \in [\rad(\fa)]^{\rf}$. Then $X \subset \sum_{Y \in \cU} Y$. Since $X$ is finitely generated, there is a finite subset $\cV$ of $\cU$ such that $X \subset \sum_{Y \in \cV} Y$. Let $k$ be such that $Y^k \subset \fa$ for all $Y \in \cV$, and let $n=k \cdot \# \cV$. Then it follows from Proposition~\ref{prop:multiplication}(a,b) that $X^n \subset \fa$.
\end{proof}

The usual relationship between radical and prime ideals holds in full generality:

\begin{proposition}
Let $\fa$ be an ideal of $A$. Then $\rad(\fa)$ is the intersection of the prime ideals containing $\fa$.
\end{proposition}

\begin{proof}
Passing to $A/\fa$, it suffices to show that $\rad(A)=\fq$, where $\fq$ is the intersection of all prime ideals of $A$. Let $X \in [A]$ satisfy $X^n=0$. Then for any prime $\fp$, we have $X^n \subset \fp$ and so $X \subset \fp$. Since $\rad(\fa)$ is the sum of such $X$, it follows that $\rad(\fa) \subset \fp$. Hence $\rad(\fa) \subset \fq$.

Now let $X \in [A]^{\rf}$ satisfy $X^n \ne 0$ for all $n$. We construct a prime $\fp$ such that $X \not\subset \fp$. Let $S$ be the set of all ideals $\fa$ such that $X^n \not\subset \fa$ for all $n$. Then $S$ is non-empty since it contains the zero ideal. Let $\{\fa_i\}_{i\in I}$ be an ascending chain in $S$ and let $\fa=\sum_{i \in I} \fa_i$ be the sum. Then $\fa$ belongs to $S$: indeed, if $X^n \subset \fa$ then we would have $X^n \subset \fa_i$ for some $i$ since $X^n$ is finitely generated, which is not the case. By Zorn's lemma, $S$ has a maximal element $\fp$. We claim $\fp$ is prime. Suppose $\fa \fb \subset \fp$ for ideals $\fa$ and $\fb$, and suppose $\fa \not\subset \fp$ and $\fb \not\subset \fp$. Then $\fp+\fa$ and $\fp+\fb$ strictly contain $\fp$, and therefore do not belong to $S$. Thus $X^n \subset \fp+\fa$ and $X^m \subset \fp+\fb$ for some $n$ and $m$, and so $X^{n+m} \subset (\fp+\fa)(\fb+\fb) \subset \fp$, a contradiction. Thus $\fa \subset \fp$ or $\fb \subset \fp$, and so $\fp$ is prime.

The result now follows. Indeed, let $X \in [\fq]^{\rf}$. Since $X$ is contained in all primes, it follows from the previous paragraph that $X^n=0$ for some $n$. Thus $X \subset \rad(\fa)$. Since this holds for all $X \in [\fq]^{\rf}$, it follows that $\fq \subset \rad(\fa)$.
\end{proof}

\subsection{Spectrum}

Let $\cC$ and $A \in \Comm(\cC)$ be as above.

\begin{definition} \label{def:spec}
The \defn{spectrum} of $A$, denoted $\Spec(A)$, is the set of all prime ideals of $A$.
\end{definition}

For an ideal $\fa$ of $A$, we let $V(\fa) \subset \Spec(A)$ be the set of prime ideals containing $\fa$. These sets have the usual properties:

\begin{proposition}
We have the following:
\begin{enumerate}
\item $V(\sum_{i \in I} \fa_i)=\bigcap_{i \in I} V(\fa_i)$.
\item $V(\fa \fb)=V(\fa \cap \fb)=V(\fa) \cup V(\fb)$.
\item $V(\fa) \subset V(\fb)$ if and only if $\rad(\fb) \subset \rad(\fa)$.
\end{enumerate}
\end{proposition}

\begin{proof}
We leave this to the reader.
\end{proof}

Thanks to the proposition, we can define a topology on $\Spec(A)$ by declaring a set to be closed if it is of the form $V(\fa)$ for some ideal $\fa$; we call this the \defn{Zariski topology}. One can show that $\Spec(A)$ is quasi-compact (this relies on the fact that $A$ is finitely generated as an $A$-module, see Remark~\ref{rmk:fgA}).

\subsection{Generic categories}

Let $\cC$ be an admissible tensor category and let $A \in \Comm(\cC)$ be a domain.

\begin{definition}
We say that an $A$-module $M$ is \defn{torsion} if for every $Y \in [M]^{\rf}$ there exists a non-zero $X \in [A]$ such that $XY=0$. We let $\Mod_A^{\tors}$ be the full subcategory of $\Mod_A$ spanned by torsion modules.
\end{definition}

Recall that a \defn{localizing subcategory} of a Grothendieck abelian category is a Serre subcategory closed under arbitrary direct sums.

\begin{proposition}
$\Mod_A^{\tors}$ is a localizing subcategory of $\Mod_A$.
\end{proposition}

\begin{proof}
Suppose that $M$ is torsion. It is clear any submodule of $M$ is torsion. Let $f \colon M \to N$ be a surjection of $A$-modules and let $Y \in [N]^{\rf}$. Let $\{Z_i\}_{i \in I}$ be the finitely generated subobjects of $f^{-1}(N)$. Since $Y$ is finitely generated, some $Z_i$ surjects onto $Y$. Let $X \in [A]$ be non-zero such that $XZ_i=0$. Then $XY=0$ as well, and so $N$ is torsion.

Next, consider a short exact sequence
\begin{displaymath}
0 \to M_1 \to M_2 \to M_3 \to 0
\end{displaymath}
where $M_1$ and $M_3$ are torsion. Let $Y \in [M_2]^{\rf}$, and let $Y'$ be its image in $M_3$. Since $Y' \in [M_3]^{\rf}$, it follows that there exists $X \in [A]$ non-zero such that $XY'=0$; of course, we can assume that $X \in [A]^{\rf}$. Thus $XY \subset [M_1]^{\rf}$. Hence there exists $X' \in [A]$ non-zero such that $X'(XY)=0$. Thus $(XX')Y=0$ and $XX' \in [A]$ is non-zero since $A$ is a domain. This shows that $M_2$ is torsion.

We have thus shown that $\Mod_A^{\tors}$ is a Serre subcategory. In particular, it is closed under finite direct sums. Let $\{M_i\}_{i \in I}$ be an arbitrary family of torsion $A$-modules, let $M=\bigoplus_{i \in I} M_i$, and let $Y \in [M]^{\rf}$. Since $Y$ is finitely generated, we have $Y \subset \bigoplus_{i \in J} M_i$ for some finite subset $J$ of $I$. Since this finite direct sum is torsion, we have $XY=0$ for some $X \in [A]$ non-zero. Thus $M$ is torsion, which completes the proof.
\end{proof}

\begin{definition}
The \defn{generic category} of $A$, denoted $\Mod_A^{\gen}$, is the Serre quotient category $\Mod_A/\Mod_A^{\tors}$.
\end{definition}

Intuitively, $\Mod_A^{\gen}$ should be thought of as the module category of the fraction field of $A$; however, the ``fraction field'' of $A$ may not actually exist as an algebra object in $\cC$. It follows from the general theory of Serre quotients that $\Mod_A^{\gen}$ is a Grothendieck abelian category and that the quotient functor $\Mod_A \to \Mod_A^{\gen}$ is cocontinuous.

\section{Fiber functors} \label{s:fiber}

\subsection{Forgetful functors}

Let $\omega \colon \cC \to \cD$ be an additive functor of Grothendieck abelian categories. We consider the following conditions:
\begin{itemize}
\item[\prop{PI}\label{PI}] For every family $\{M_i\}_{i \in I}$ of objects of $\cC$, the natural map $\omega(\prod_{i \in I} M_i ) \to \prod_{i \in I} \omega(M_i)$ is injective.
\item[\prop{For}\label{For}] $\omega$ is exact, faithful, cocontinuous, and satisfies \prop{PI}.
\end{itemize}
A typical example of a functor satisfying \prop{For} is the forgetful functor from the category of graded vector spaces to the category of vector spaces; note that this functor does not commute with products, but does satisfy \prop{PI}.

Fix $\omega \colon \cC \to \cD$ satisfying \prop{For}. We then think of $\omega$ as a kind of forgetful functor, and often write $M^{\omega}$ in place of $\omega(M)$.

\begin{proposition} \label{prop:omega-zero}
Let $M$ be an object of $\cC$ such that $M^{\omega}=0$. Then $M=0$.
\end{proposition}

\begin{proof}
We have $\id_M^{\omega}=0$, and so $\id_M=0$ since $\omega$ is faithful. Thus $M=0$.
\end{proof}

\begin{proposition}
Let $f \colon M \to N$ be a morphism in $\cC$. Then $f$ is injective (resp.\ surjective) if and only if $f^{\omega}$ is injective (resp.\ surjective).
\end{proposition}

\begin{proof}
If $f$ is injective or surjective, so is $f^{\omega}$ since $\omega$ is exact. Now suppose that $f^{\omega}$ is injective. Then $\ker(f)^{\omega}=\ker(f^{\omega})=0$, and so $\ker(f)=0$ (Proposition~\ref{prop:omega-zero}), and so $f$ is injective. The proof for surjectivity is similar.
\end{proof}

\begin{proposition} \label{prop:for-contain}
Let $M$ be an object of $\cC$ and let $X$ and $Y$ be subobjects of $M$.
\begin{enumerate}
\item $X^{\omega}$ is a subobject of $M^{\omega}$.
\item $X \subset Y$ if and only if $X^{\omega} \subset Y^{\omega}$.
\item $X=Y$ if and only if $X^{\omega}=Y^{\omega}$.
\end{enumerate}
\end{proposition}

\begin{proof}
(a) follows from the exactness of $\omega$. If $X \subset Y$ then clearly $X^{\omega} \subset Y^{\omega}$. Conversely, suppose that $X^{\omega} \subset Y^{\omega}$. Let $Z=X+Y$. Then $Z^{\omega}=X^{\omega}+Y^{\omega}=Y^{\omega}$. We thus see that $(Z/Y)^{\omega}=0$, and so $Y=Z$ (Proposition~\ref{prop:omega-zero}), and so $X \subset Y$. This proves (b), and (c) follows from (b).
\end{proof}

\begin{proposition} \label{prop:sum-int}
Let $M$ be an object of $\cC$ and let $\{N_i\}_{i \in I}$ be a family of subobjects. Then
\begin{displaymath}
\big( \sum_{i \in I} N_i \big)^{\omega} = \sum_{i \in I} N_i^{\omega}, \qquad
\big( \bigcap_{i \in I} N_i \big)^{\omega} = \bigcap_{i \in I} N_i^{\omega}.
\end{displaymath}
\end{proposition}

\begin{proof}
Let $S=\sum_{I \in I} N_i$. Then $S$ is the image of the map $\bigoplus_{i \in I} N_i \to M$. We thus see that $S^{\omega}$ is the image of the map $\bigoplus_{i \in I} N_i^{\omega} \to M^{\omega}$, which is just $\sum_{i \in I} N_i^{\omega}$. Note that here we used the fact that $\omega$ commutes with direct sums.

Now let $P=\bigcap_{i \in I} N_i$. Then $P$ is the kernel of the map $f \colon M \to \prod_{i \in I} M/N_i$. We thus see that $P^{\omega}=\ker(f^{\omega})$. Consider the maps
\begin{displaymath}
\xymatrix@C=3em{
M^{\omega} \ar[r]^-{f^{\omega}} & (\prod_{i \in I} M/N_i)^{\omega} \ar[r]^-g & \prod_{i \in I} (M/N_i)^{\omega} }
\end{displaymath}
where $g$ is the natural map. By \prop{PI}, $g$ is injective, and so we have
\begin{displaymath}
P^{\omega}=\ker(f^{\omega})=\ker(g \circ f^{\omega}) = \bigcap_{i \in I} N_i^{\omega}.
\end{displaymath}
This completes the proof.
\end{proof}

\begin{proposition} \label{prop:min-max}
Let $M$ be an object of $\cC$ and let $N$ be a subobject of $M^{\omega}$.
\begin{enumerate}
\item There exists a unique maximal subobject $X$ of $M$ such that $X^{\omega} \subset N$.
\item There exists a unique minimal subobject $Y$ of $M$ such that $N \subset Y^{\omega}$.
\end{enumerate}
\end{proposition}

\begin{proof}
(a) Let $\cU$ be the set of all subobjects $T$ of $M$ such that $T^{\omega} \subset N$, and let $X=\sum_{T \in \cU} T$. By Proposition~\ref{prop:sum-int}, we have $X^{\omega}=\sum_{T \in \cU} T^{\omega} \subset N$, and so $X \in \cU$. Clearly, $X$ is the unique maximal member of $\cU$.

(b) Let $\cV$ be the set of all subobjects $T$ of $M$ such that $N \subset T^{\omega}$, and let $Y=\bigcap_{T \in \cU} T$. By Proposition~\ref{prop:sum-int}, we have $Y^{\omega}=\bigcap_{T \in \cV} T^{\omega} \supset N$, and so $Y \in \cV$. Clearly, $Y$ is the unique minimal member of $\cV$.
\end{proof}

\begin{definition}
Let $M$ be an object of $\cC$ and let $N$ be a subobject of $M^{\omega}$.
\begin{enumerate}
\item We write $\lfloor N \rfloor_{\cC}$, or simply $\lfloor N \rfloor$, for the maximal object in Proposition~\ref{prop:min-max}(a).
\item We write $\lceil N \rceil_{\cC}$, or simply $\lceil N \rceil$, for the minimal object in Proposition~\ref{prop:min-max}(b). \qedhere
\end{enumerate}
\end{definition}

\begin{example}
Let $G$ be a group, let $k$ be a field, let $\cC=\Rep_k(G)$, let $\cD=\Vec_k$, and let $\omega \colon \cC \to \cD$ be the forgetful functor. Let $M \in \cC$ and let $N$ be a subobject of $M^{\omega}$; thus $M$ is a representation of $G$, and $N$ is a vector subspace of $M$. In this case, $\lfloor N \rfloor = \bigcap_{g \in G} gN$ is the maximal subrepresentation contained in $N$, and $\lceil N \rceil = \sum_{g \in G} gN$ is the subrepresentation generated by $N$.
\end{example}

\begin{proposition} \label{prop:floor-int}
Let $M$ be an object of $\cC$ and let $\{N_i\}_{i \in I}$ be a family of subobjects of $M^{\omega}$. Then
\begin{displaymath}
\big\lfloor \bigcap_{i \in I} N_i \big\rfloor = \bigcap_{i \in I} \lfloor N_i \rfloor, \qquad
\big\lceil \sum_{i \in I} N_i \big\rceil = \sum_{i \in I} \lceil N_i \rceil.
\end{displaymath}
\end{proposition}

\begin{proof}
Let $X=\lfloor \bigcap_{i \in I} N_i \rfloor$ and $Y=\bigcap_{i \in I} \lfloor N_i \rfloor$. We have $X^{\omega} \subset \bigcap_{i \in I} N_i$, and so $X \subset N_i$ for each $i$. Thus, by definition, we have $X \subset \lfloor N_i \rfloor$ for each $i$, and so $X \subset Y$. By Proposition~\ref{prop:sum-int}, we have $Y^{\omega}=\bigcap_{i \in I} \lfloor N_i \rfloor^{\omega} \subset \bigcap_{i \in I} N_i$. Thus, by definition, we have $Y \subset X$. The proof for sums is similar.
\end{proof}

\begin{proposition} \label{prop:ceil-fg}
Let $M$ be an object of $\cC$, and let $N$ be a finitely generated subobject of $M^{\omega}$. Then $\lceil N \rceil$ is finitely generated.
\end{proposition}

\begin{proof}
Suppose that $\lceil N \rceil=\sum_{i \in I} K_i$ for subobjects $K_i$ of $M$. Appealing to Proposition~\ref{prop:sum-int}, we find $N \subset \lceil N \rceil^{\omega} = \sum_{i \in I} K_i^{\omega}$. Since $N$ is finitely generated, it follows that there is a finite subset $J$ of $I$ such that $N \subset \sum_{i \in J} K_i^{\omega} = (\sum_{i \in J} K_i)^{\omega}$. Thus, by definition, we have $\lceil N \rceil \subset \sum_{i \in J} K_i$, and so $\lceil N \rceil$ is finitely generated.
\end{proof}

\subsection{Fiber functors}

In the theory of Tannakian categories, a fiber functor is a symmetric monoidal  functor to vector spaces that is exact and faithful. We will use the term in a slightly different sense:

\begin{definition}
Let $\cC$ and $\cD$ be tensor categories. A \defn{fiber functor} $\omega \colon \cC \to \cD$ is a symmetric monoidal functor that satisfies \prop{For}.
\end{definition}

Fix a fiber functor $\omega \colon \cC \to \cD$ and $A \in \Comm(\cC)$.

\begin{proposition}
Let $M$ be an $A$-module, and let $X \subset A$ and $Y \subset M$ be subobjects. Then $(XY)^{\omega}=X^{\omega} Y^{\omega}$.
\end{proposition}

\begin{proof}
By definition, $XY$ is the image of the map $X \otimes Y \to M$. We thus see that $(XY)^{\omega}$ is the image of the map $X^{\omega} \otimes Y^{\omega}=(X \otimes Y)^{\omega} \to M^{\omega}$, which is $X^{\omega} Y^{\omega}$.
\end{proof}

\begin{proposition}
Suppose $\fp$ is an ideal of $A$ such that $\fp^{\omega}$ is a prime ideal of $A^{\omega}$. Then $\fp$ is a prime ideal of $A$.
\end{proposition}

\begin{proof}
Let $X$ and $Y$ be subobjects of $A$ such that $XY \subset \fp$. Then $X^{\omega} Y^{\omega} \subset \fp^{\omega}$. Since $\fp^{\omega}$ is prime, it follows that $X^{\omega} \subset \fp^{\omega}$ or $Y^{\omega} \subset \fp^{\omega}$. By Proposition~\ref{prop:for-contain}, we find $X \subset \fp$ or $Y \subset \fp$, and so $\fp$ is prime.
\end{proof}

\begin{proposition} \label{prop:fiber-rad}
Let $\fa$ be an ideal of $A$. Then $(\rad{\fa})^{\omega} \subset \rad(\fa^{\omega})$.
\end{proposition}

\begin{proof}
Write $\rad{\fa}=\sum_{i \in I} X_i$ where each $X_i$ is finitely generated. Then $(\rad{\fa})^{\omega}=\sum_{i \in I} X_i^{\omega}$ by Proposition~\ref{prop:sum-int}. It thus suffices to show that $X_i^{\omega} \subset \rad(\fa^{\omega})$ for each $i$. Thus fix $i \in I$. Since $X_i$ is finitely generated and contained in $\rad(\fa)$, we have $X_i^n \subset \fa$ for some $n$. We thus see that $(X_i^{\omega})^n=(X_i^n)^{\omega} \subset \fa^{\omega}$, and so $X_i^{\omega} \subset \rad(\fa^{\omega})$. This completes the proof.
\end{proof}

\begin{proposition}
Let $M$ be an $A$-module, and let $N$ be an $A^{\omega}$-submodule of $M^{\omega}$. Then $\lfloor N \rfloor$ is an $A$-submodule of $M$.
\end{proposition}

\begin{proof}
We have
\begin{displaymath}
(A \lfloor N \rfloor)^{\omega}=A^{\omega} (\lfloor N \rfloor)^{\omega} \subset A^{\omega} N \subset N,
\end{displaymath}
and so $A \lfloor N \rfloor \subset \lfloor N \rfloor$, i.e., $\lfloor N \rfloor$ is an $A$-submodule of $M$.
\end{proof}

In particular, we see that if $\fa$ is an ideal of $A^{\omega}$ then $\lfloor \fa \rfloor$ is an ideal of $A$. This construction should be thought of as a kind of contraction. The following provides evidence for this:

\begin{proposition} \label{prop:contract-prime}
Let $\fq$ be a prime ideal of $A^{\omega}$. Then $\lfloor \fq \rfloor$ is a prime ideal of $A$.
\end{proposition}

\begin{proof}
Let $X$ and $Y$ be subobjects of $A$ such that $XY \subset \lfloor \fq \rfloor$. Thus $X^{\omega} Y^{\omega}=(XY)^{\omega} \subset \fq$. Since $\fq$ is prime, we see that $X^{\omega} \subset \fq$ or $Y^{\omega} \subset \fq$. Thus $X \subset \lfloor \fq \rfloor$ or $Y \subset \lfloor \fq \rfloor$, and so $\lfloor \fq \rfloor$ is prime.
\end{proof}

We thus have a function $\Spec(A^{\omega}) \to \Spec(A)$ given by $\fq \mapsto \lfloor \fq \rfloor$. One easily sees that it is continuous.

\section{Abstract comparison results} \label{s:axioms}

\subsection{The main properties} \label{ss:mainprops}

Let $\omega \colon \cC \to \cD$ be a fiber functor of admissible tensor categories. We consider the following conditions:
\begin{itemize}
\item[\prop{A}\label{A}] Let $A \in \Comm(\cC)$ and let $\fa$ and $\fb$ be ideals of $A$. Then $\rad(\fa) \subset \rad(\fb)$ if and only if $\rad(\fa^{\omega}) \subset \rad(\fb^{\omega})$.
\item[\prop{B}\label{B}] Let $A \in \Comm(\cC)$ and let $\fa$ be an ideal of $A$. Then $\rad(\fa)$ is prime if and only if $\rad(\fa^{\omega})$ is prime.
\end{itemize}
These are properties that are both useful and reasonable to expect in many situations. We observe that these properties imply that $\omega$ behaves well with respect to minimal primes (improving on \cite[Theorem~C]{tcaprimes}):

\begin{proposition}
Suppose $\omega$ satisfies \prop{A} and \prop{B}, and let $A \in \Comm(\cC)$. Then we have mutually inverse bijections
\begin{displaymath}
\xymatrix{
\{ \text{minimal primes of $A$} \} \ar@<2pt>[r]^{\Phi} &
\{ \text{minimal primes of $A^{\omega}$} \} \ar@<2pt>[l]^{\Psi} }
\end{displaymath}
given by $\Phi(\fp)=\rad(\fp^{\omega})$ and $\Psi(\fq)=\lfloor \fq \rfloor$.
\end{proposition}

\begin{proof}
For any ideal $\fa$ of $A$, let $\Phi(\fa)=\rad(\fa^{\omega})$, and for any ideal $\fb$ of $A^{\omega}$ let $\Psi(\fb)=\lfloor \fb \rfloor$. Then $\Phi$ takes primes to primes by \prop{B}, and $\Psi$ takes primes to primes by Proposition~\ref{prop:contract-prime}.

Let $\fq$ be a minimal prime of $A^{\omega}$. Let $\fp=\Psi(\fq)$, and let $\fp_0 \subset \fp$ be a minimal prime. Since $\fp^{\omega} \subset \fq$, we have $\rad(\fp^{\omega}) \subset \fq$. Thus $\Phi(\fp_0) \subset \Phi(\fp) \subset \fq$, and so all three coincide by the minimality of $\fq$. By \prop{A}, we see that $\rad(\fp_0)=\rad(\fp)$, and so $\fp=\fp_0$, i.e., $\fp$ is a minimal prime. Thus $\Psi$ maps minimal primes to minimal primes, and $\Phi \circ \Psi$ is the identity on minimal primes. In particular, $\Phi$ is surjective on minimal primes.

Now let $\fp$ be a minimal prime of $A$. Let $\fq \subset \Phi(\fp)$ be a minimal prime. By the previous paragraph, we have $\fq=\Phi(\fp')$ for some minimal prime $\fp'$ of $A$. The containment $\Phi(\fp') \subset \Phi(\fp)$ implies $\fp' \subset \fp$, by \prop{A}, and so $\fp=\fp'$ by the minimality of $\fp$. Thus $\Phi(\fp)$ is a minimal prime. Since $\Psi$ maps minimal primes to minimal primes, we see that $\Psi(\Phi(\fp))$ is a minimal prime of $A$. Since $\fp^{\omega} \subset \rad(\fp^{\omega})=\Phi(\fp)$, we have $\fp \subset \Psi(\Phi(\fp))$. By minimality of $\Psi(\Phi(\fp))$, we must have equality. Thus $\Psi \circ \Phi$ is the identity on minimal primes, which completes the proof.
\end{proof}

\subsection{Generalities on property (A)}

Given $A \in \Comm(\cC)$, we say that $X \in [A]$ is \defn{locally nilpotent} if $X \subset \rad(A)$. Consider the following property on $\omega \colon \cC \to \cD$:
\begin{itemize}
\item[\propn{A}{1}] Let $A \in \Comm(\cC)$ and let $X \in [A]$. Then $X$ is locally nilpotent if and only if $X^{\omega}$ is locally nilpotent.
\item[\propn{A}{2}] For $A \in \Comm(\cC)$ we have $\lfloor \rad(A^{\omega}) \rfloor = \rad(A)$.
\end{itemize}

\begin{proposition}
We have $\prop{A} \iff \propn{A}{1} \iff \propn{A}{2}$.
\end{proposition}

\begin{proof}
Fix $A \in \Comm(\cC)$ throughout the proof.

Assume \prop{A}. Let $X \in [A]$. Applying \prop{A} with $\fa$ being the ideal generated by $X$ and $\fb=(0)$, we see that $X$ is locally nilpotent if and only if $X^{\omega}$ is. Thus \propn{A}{1} holds.

Now suppose that \propn{A}{1} holds. Then an ideal $\fa$ is locally nilpotent if and only if $\fa^{\omega}$ is; thus $\fa \subset \rad(A)$ if and only if $\fa^{\omega} \subset \rad(A^{\omega})$. Now let $\fa$ and $\fb$ be two ideals of $A$. Let $\ol{\fa}$ be the image of $\fa$ in $\ol{A}=A/\fb$. Then $\ol{\fa} \subset \rad(\ol{A})$ if and only if $\ol{\fa}^{\omega} \subset \rad(\ol{A}{}^{\omega})$, and so $\fa+\fb \subset \rad(\fb)$ if and only if $\fa^{\omega}+\fb^{\omega} \subset \rad(\fb^{\omega})$. It follows that $\fa \subset \rad(\fb)$ if and only if $\fa^{\omega} \subset \rad(\fb^{\omega})$, and so $\rad(\fa) \subset \rad(\fb)$ if and only if $\rad(\fa^{\omega}) \subset \rad(\fb^{\omega})$. Thus \prop{A} holds.

Finally, suppose $X \in [A]$. Then $X$ is locally nilpotent if and only if $X \subset \rad(A)$, and $X^{\omega}$ is locally nilpotent if and only if $X^{\omega} \subset \rad(A^{\omega})$, which in turn is equivalent to $X \subset \lfloor \rad(A^{\omega}) \rfloor$. We thus see that \propn{A}{1} is equivalent to \propn{A}{2}.
\end{proof}

Consider the following property:
\begin{itemize}
\item[\prop{Fin}\label{Fin}] $\omega$ carries finitely generated objects to finitely generated objects.
\end{itemize}

\begin{proposition} \label{prop:FinA}
We have $\prop{Fin} \implies \prop{A}$.
\end{proposition}

\begin{proof}
Let $A \in \Comm(\cC)$ and let $X \in [A]$ be a subobject with $X^{\omega}$ locally nilpotent. Write $X=\sum_{i \in I} X_i$ where each $X_i$ is finitely generated. Then $X_i^{\omega}$ is finitely generated, by \prop{Fin}, and locally nilpotent, and thus nilpotent (Proposition~\ref{prop:fg-nilp}). Hence $X_i$ is nilpotent too, and so $X$ is locally nilpotent.
\end{proof}

\begin{proposition} \label{prop:compA}
Suppose $\omega \colon \cC \to \cD$ and $\eta \colon \cD \to \cE$ are fiber functors satisfying \prop{A}. Then $\eta \circ \omega$ also satisfies \prop{A}.
\end{proposition}

\begin{proof}
Let $A \in \Comm(\cC)$ and let $X \in [A]$. Then $X$ is locally nilpotent if and only if $X^{\omega}$ is locally nilpotent, since $\omega$ satisfies \propn{A}{1}. Similarly, $X^{\omega}$ is locally nilpotent if and only if $(X^{\omega})^{\eta}$ is locally nilpotent, since $\eta$ satisfies \propn{A}{1}. We thus see that $X$ is locally nilpotent if and only if $(X^{\omega})^{\eta}$ is, and so $\eta \circ \omega$ satisfies \propn{A}{1}, and thus \prop{A}.
\end{proof}

\subsection{Property (A) for generic categories}

Consider the following property on $\omega \colon \cC \to \cD$:
\begin{itemize}
\item[\prop{Gen}\label{Gen}] There exists an integral algebra $R \in \Comm(\cC)$ and an ideal $\fm \subset R^{\omega}$ such that
\begin{enumerate}[label=(\roman*)]
\item the natural map $\bone_{\cD} \to R^{\omega}/\fm$ is an isomorphism
\item the functor $\Mod_R \to \cD$ given by $M \mapsto M^{\omega}/\fm M^{\omega}$ is exact and kills $\Mod_R^{\tors}$
\item the induced functor $\Mod_R^{\gen} \to \cD$ is an equivalence
\end{enumerate}
\end{itemize}
Suppose $\eta \colon \cD \to \cE$ is a second fiber functor. We consider the following property:
\begin{itemize}
\item[\prop{Rad}\label{prop: Rad}] Let $A \in \Comm(\cC)$, let $\fb$ be an ideal of $A^{\omega}$, and let $\fa=\lfloor \fb \rfloor_{\cC}$. Then we have $\lfloor \rad(\fb^{\eta}) \rfloor_{\cC}=\lfloor \rad(\fa^{\omega \eta}) \rfloor_{\cC}$.
\end{itemize}

\begin{example}
Suppose $H \subset G$ are groups, $\omega \colon \Rep(G) \to \Rep(H)$ is the restriction functor, and $\eta \colon \Rep(H) \to \Vec$ is the forgetful functor. Then \prop{Rad} says the following: if $A$ is an algebra on which $G$ acts and $\fb$ is an $H$-stable ideal then $\bigcap_{g \in G} g \rad(\fb) = \rad(\bigcap_{g \in G} g \fb)$, where here $\rad$ means the ordinary radical.
\end{example}

\begin{proposition} \label{prop:Agen}
Suppose \prop{Rad} holds, $\omega$ satisfies \prop{Gen}, and $\eta \circ \omega$ satisfies \prop{A}. Then $\eta$ satisfies \prop{A}.
\end{proposition}

\begin{proof}
Let $B \in \Comm(\cD)$ and let $X \in [B]$ be a subobject with $X^{\eta}$ locally nilpotent; thus $X \subset \lfloor \rad(B^{\eta}) \rfloor$. Choose a commutative $R$-algebra $A$ in $\cC$ such that $B\cong A^{\omega}/\fm A^{\omega}$ and $A$ is $R$-torsionfree. (One can construct $A$ by lifting a presentation for $B$ and killing torsion.) Let $Y$ be an $R$-submodule of $A$ such that $Y^{\omega}/\fm Y^{\omega}=X$. Let $\fb=\fm A^{\omega}$ and let $\fa=\lfloor \fb \rfloor_{\cC}$. Then $Y^{\omega \eta} \subset \rad(\fb^{\eta})$, and so $Y \subset \lfloor \rad(\fb^{\eta}) \rfloor_{\cC}$. Thus by \prop{Rad}, we have $Y \subset \lfloor \rad(\fa^{\omega \eta}) \rfloor_{\cC}$. Since $\eta \circ \omega$ satisfies \prop{A}, this means $Y \subset \rad(\fa)$. We claim that $\fa=0$. Indeed, we have $\fa^{\omega} \subset \fm A^{\omega}$, and so the map $\fa^{\omega}/\fm \fa^{\omega} \to A/\fm A^{\omega}$ is the zero map. By \prop{Gen}, this means that the map $\fa \to A$ is zero in $\Mod_R^{\gen}$. Since $A$ is $R$-torsionfree, it follows that $\fa=0$, as claimed. We thus see that $Y \subset \rad(A)$, and so $X \subset \rad(B)$, as required.
\end{proof}

\subsection{Generalities on property (B)}

Consider the following property:
\begin{itemize}
\item[(B${}_1$)] Let $A \in \Comm(\cC)$ and let $X,Y \in [A^{\omega}]^{\rf}$ satisfy $XY=0$. Then $\lceil X^n \rceil \lceil Y \rceil=0$ for some $n \ge 1$.
\end{itemize}

\begin{proposition}
Let $A \in \Comm(\cC)$ and let $\fa$ be an ideal of $A$.
\begin{enumerate}
\item Suppose $\omega$ satisfies \prop{A} and $\rad(\fa^{\omega})$ is prime. Then $\rad(\fa)$ is prime.
\item Suppose $\omega$ satisfies \propn{B}{1} and $\rad(\fa)$ is prime. Then $\rad(\fa^{\omega})$ is prime.
\end{enumerate}
\end{proposition}

\begin{proof}
Passing to $A/\fa$, it suffices to treat the case where $\fa=0$.

(a) Let $XY \subset \rad(A)$ with $X$ and $Y$ subobjects of $A$. Then $X^{\omega} Y^{\omega} \subset (\rad{A})^{\omega} \subset \rad(A^{\omega})$, where the second containment comes from Proposition~\ref{prop:fiber-rad}. Since $\rad(A^{\omega})$ is prime, we have $X^{\omega} \subset \rad(A^{\omega})$ or $Y^{\omega} \subset \rad(A^{\omega})$. Suppose the former holds. Then $X^{\omega}$ is locally nilpotent, and so $X$ is locally nilpotent by \prop{A}. Thus $X \subset \rad(A)$, and so $\rad(A)$ is prime.

(b) Let $XY \subset \rad(A^{\omega})$ with $X, Y \in [A^{\omega}]^{\rf}$. Then $X^kY^k=0$ for some $k \ge 1$. By \propn{B}{1}, we find $\lceil X^{nk} \rceil \lceil Y^k \rceil=0$ for some $n \ge 1$. Since $\rad(A)$ is prime, we see that $\lceil X^{nk} \rceil$ or $\lceil Y^k \rceil$ is nilpotent (note that both are fintiely generated by Proposition~\ref{prop:ceil-fg}). Thus either $X$ or $Y$ is nilpotent, and so $X \subset \rad(A^{\omega})$ or $Y \subset \rad(A^{\omega})$. Thus $\rad(A^{\omega})$ is prime.
\end{proof}

\begin{corollary}
We have $\prop{A} \land \propn{B}{1} \implies \prop{B}$.
\end{corollary}

\section{Abstract results for Lie algebras} \label{s:lie}

\subsection{Representations of Lie algebras in tensor categories}

Let $\cV$ be an admissible tensor category and let $\fg$ be a Lie algebra in $\cV$. We can then consider the category $\Rep(\fg)$ of representations of $\fg$ in $\cV$. This admits a natural tensor product, and is a tensor category in the sense of Definition~\ref{def:tensor-cat}. (It is not necessarily admissible, see Example~\ref{ex:inadmissible}.) We are typically interested in subcategories of $\Rep(\fg)$. Since we want to work with admissible categories, we introduce the following notion:

\begin{definition} \label{def:admissible}
An \defn{admissible subcategory} of $\Rep(\fg)$ is a full subcategory $\cC$ satisfying the following conditions:
\begin{enumerate}
\item If $M \in \cC$ then any subquotient of $M$ in $\Rep(\fg)$ belongs to $\cC$.
\item $\cC$ is closed under arbitrary direct sums.
\item $\cC$ is closed under tensor products.
\item If $M$ and $N$ are finitely generated $\fg$-modules that belong to $\cC$ then $M \otimes N$ is finitely generated as a $\fg$-module. \qedhere
\end{enumerate}
\end{definition}

\begin{proposition}
Let $\cC$ be an admissible subcategory of $\Rep(\fg)$. Then $\cC$ (with the induced tensor product) is an admissible tensor category. Furthermore, an object of $\cC$ is finitely generated if and only if it finitely generated as a $\fg$-module.
\end{proposition}

\begin{proof}
By~(a) and (b), $\cC$ is an abelian subcategory of $\Rep(\fg)$, and is cocomplete; furthermore, colimits in $\cC$ can be computed in $\Rep(\fg)$, and so filtered colimits in $\cC$ are exact. Let $E$ be a generator for $\cC$, and let $\{F_i\}_{i \in I}$ be the set of $\fg$-module quotients of $\cU(\fg) \otimes E$ that belong to $\cC$. We claim that $\{F_i\}$ is a generating set for $\cC$. Let $M$ be an object of $\cC$ and let $K$ be a proper subobject. We can then find a map $E \to M$ in $\cC$ with image not contained in $K$. We thus get a map $\cU(\fg) \otimes E \to M$ of $\fg$-modules with image not contained in $K$. The image of this map is isomorphic to some $F_i$. Thus we have a map $F_i \to M$ with image not contained in $K$, which verifies the claim. We thus see that $\cC$ is a Grothendieck abelian category. Since colimits and tensor products in $\cC$ can be computed in $\Rep(\fg)$, it follows that $\otimes$ is cocontinuous in each variable on $\cC$. Thus $\cC$ is a tensor category.

Let $M$ be an object of $\cC$. By (a), the subobjects of $M$ in $\cC$ are the same as the subobjects of $M$ in $\Rep(\fg)$. Since finite generation is defined in terms of subobjects, we see that $M$ is finitely generated in $\cC$ if and only if it is finitely generated in $\Rep(\fg)$. Since $M$ is the sum of its finitely generated subobjects in $\Rep(\fg)$, it follows that the same holds in $\cC$. By (d), we see that the tensor product of two finitely generated subobjects of $\cC$ is again finitely generated. We thus see that $\cC$ is an admissible tensor category.
\end{proof}

\begin{proposition}
Let $\cC$ be an admissible subcategory of $\Rep(\fg)$. Then the inclusion $\cC \to \Rep(\fg)$ satisfies \prop{PI}.
\end{proposition}

\begin{proof}
Let $i \colon \cC \to \Rep(\fg)$ be the inclusion functor and let $j \colon \Rep(\fg) \to \cC$ be the functor assigning to each $\fg$-module the maximum submodule that belongs to $\cC$. Then $i$ is left adjoint to $j$. (This shows that $\cC$ is a ``mono-coreflexive'' subcategory of $\Rep(\fg$).) Let $\{M_i\}_{i \in I}$ be a family of objects in $\cC$. Let $\prod M_i$ be their product in $\Rep(\fg)$ and $\prod^{\cC} M_i$ their product in $\cC$. Since $j$ is is continuous, it preserves products, and so $\prod^{\cC} M_i=j(\prod M_i)$. We thus see that $\prod^{\cC} M_i$ is a subobject of $\prod M_i$ in $\Rep(\fg)$, and so the natural map $\prod^{\cC} M_i \to \prod M_i$ is injective.
\end{proof}

\begin{definition}
Let $\fh$ be a Lie subalgebra of $\fg$. Let $\cC \subset \Rep(\fg)$ and $\cD \subset \Rep(\fh)$ be admissible subcategories. We say that $\cC$ and $\cD$ are \defn{compatible} if the restriction functor $\Rep(\fg) \to \Rep(\fh)$ carries $\cC$ into $\cD$.
\end{definition}

\begin{proposition}
Let $\fh \subset \fg$ be Lie algebras and let $\cC \subset \Rep(\fg)$ and $\cD \subset \Rep(\fh)$ be compatible admissible subcategories. Then the restriction functor $\omega \colon \cC \to \cD$ is a fiber functor.
\end{proposition}

\begin{proof}
It is clear that $\omega$ is exact and faithful. Since direct sums in $\cC$ and $\cD$ can be computed in the underlying category $\cV$, it follows that $\omega$ is compatible with direct sums, and therefore cocontinuous. Let $\{M_i\}_{i \in I}$ be a family of objects in $\cC$. Then we have the following commutative triangle in $\cV$
\begin{displaymath}
\xymatrix{
& \prod^{\cV} M_i \\
\prod^{\cC} M_i \ar[ru] \ar[rr] && \prod^{\cD} M_i \ar[lu] }
\end{displaymath}
where the superscripts denote the category in which the product is formed. By the definition of admissible subcategory, the vertical morphisms are injective. It follows that the horizontal morphism is injective, and so $\omega$ satisfies \prop{PI}. We thus see that $\omega$ satisfies \prop{For}. Since tensor products in $\cC$ and $\cD$ are computed in $\cV$, it follows that $\omega$ is naturally a symmetric monoidal functor. Thus $\omega$ is a fiber functor.
\end{proof}

Taking $\fh=0$ and $\cD=\cV$, we obtain the following:

\begin{corollary}
Let $\cC \subset \Rep(\fg)$ be an admissible subcategory. Then the forgetful functor $\cC \to \cV$ is a fiber functor.
\end{corollary}

\begin{definition}
Let $\cC$ be an admissible subcategory of $\Rep(\fg)$. We say that $\cC$ satisfies \prop{A} or \prop{B} (or any other property of fiber functors) if the forgetful functor $\cC \to \cV$ does.
\end{definition}

\begin{proposition}
Let $\fh \subset \fg$ be Lie algebras and let $\cC \subset \Rep(\fg)$ and $\cD \subset \Rep(\fh)$ be compatible admissible subcategories. Let $M \in \cC$ and let $N$ be an $\fh$-submodule of $M$.
\begin{enumerate}
\item We have $\lceil N \rceil = \cU(\fg) N$.
\item $\lfloor N \rfloor$ is the sum of all $\fg$-submodules of $M$ contained in $N$. If $X$ is a subobject of $M$ then $X \subset \lfloor N \rfloor$ if and only if $\cU(\fg) X \subset N$.
\end{enumerate}
\end{proposition}

\begin{proof}
(a) We have $N \subset \cU(\fg) N$, and so $\lceil N \rceil \subset \cU(\fg) N$. If $K$ is any $\fg$-submodule of $M$ containing $N$ then $K$ contains $\cU(\fg) N$; thus $\cU(\fg) N \subset \lceil N \rceil$.

(b) The first statement is exactly how we constructed $\lfloor N \rfloor$. As for the second, if $X$ is contained in $\lfloor N \rfloor$ then $\cU(\fg) X \subset \lfloor N \rfloor \subset N$, the first inclusion coming from the fact that $\lfloor N \rfloor$ is a $\fg$-submodule. Conversely, if $\cU(\fg) X \subset N$ then $\cU(\fg) X$ is a $\fg$-submodule of $M$ contained in $N$, and so $X \subset \cU(\fg) X \subset \lfloor N \rfloor$.
\end{proof}

\begin{example} \label{ex:inadmissible}
Let $k$ be a field and let $\cV=\Vec_k$ be the category of vector spaces. Let $\fg=k$ be the one-dimensional abelian Lie algebra over $k$. Then $\cU(\fg)=k[x]$, and so $\fg$-modules are the same as $k[x]$-modules. Of course, $\cU(\fg)$ is a finitely generated module, but $\cU(\fg) \otimes \cU(\fg) \cong k[x,y]$ is not. Thus $\Rep(\fg)$ is not admissible.
\end{example}

\subsection{Property (A)} \label{ss:propA}

Let $\fh \subset \fg$ be Lie algebras in $\cV$. Let $\cC \subset \Rep(\fg)$ and $\cD \subset \Rep(\fh)$ be compatible admissible subcategories, let $\omega \colon \cC \to \cD$ be the restriction functor, and let $\eta \colon \cD \to \cV$ be the forgetful functor. We suppose the following condition holds:
\begin{itemize}[leftmargin=4em]
\item[\prop{Stab}\label{prop: Stab}] Let $M \in \cC$ and let $X \subset M$ be a finitely generated $\cV$-subobject. Then there exists a Lie subalgebra $\fp$ of $\fg$ such that $\fg=\fp+\fh$, and the $\fp$-submodule of $M$ generated by $X$ is finitely generated as a $\cV$-object.
\end{itemize}
In the above condition, the algebra $\fp$ can be thought of as an ``approximate stabilizer,'' and so the condition can be intrepreted as saying that elements of $M$ have sufficiently large stabilizer relative to $\fh$.

\begin{proposition}
We have $\prop{Stab} \implies \prop{Rad}$.
\end{proposition}

\begin{proof}
Let $A \in \Comm(\cC)$; we think of $A$ as a commutative algebra in $\cV$ on which $\fg$ acts by derivations. Let $\fb$ be an ideal of $A^{\omega}$, i.e., an ideal of $A$ that if $\fh$-stable, and let $\fa=\lfloor \fb \rfloor$. Thus $\fa$ is the maximal $\fg$-submodule of $\fb$. We must show that $\lfloor \rad(\fb) \rfloor_{\cC}=\lfloor \rad(\fa) \rfloor_{\cC}$, where on both sides $\rad$ is computed in $\cV$. Explicitly, this means that if $V$ is a $\cC$-subobject of $A$ then $V \subset \rad(\fa)$ if and only if $\cU(\fg) V \subset \rad(\fb)$. Of course, it suffices to treat the case where $V$ is finitely generated.

First suppose that $V \subset \rad(\fa)$. Then $\cU(\fg) V \subset \rad(\fa)$ since $\rad(\fa)$ is $\fg$-stable. Since $\rad(\fa) \subset \rad(\fb)$, we see that $\cU(\fg) V \subset \rad(\fb)$, as required.

Now suppose that $\cU(\fg) V \subset \rad(\fb)$, with $V$ finitely generated. Per \prop{Stab}, write $\fg=\fh+\fp$ where $\fp$ is a Lie subalgebra of $\fg$ such that $W=\cU(\fp) V$ is $\cV$-finite. Since $W \subset \cU(\fg) V \subset \rad(\fb)$ and $W$ is $\cV$-finite, we have $W^n \subset \fb$ for some $n$. We thus find $\cU(\fp) W^n \subset \fb$, since $W^n$ is $\fp$-stable, and so $\cU(\fh) \cU(\fp) W^n \subset \fb$, since $\fb$ is $\fh$-stable. Since $\cU(\fg)=\cU(\fh) \cU(\fp)$, we see that $\cU(\fg) W^n \subset \fb$, and so $W^n \subset \fa$, and so $W \subset \rad(\fa)$. Since $V \subset W$, we find $V \subset \rad(\fa)$, which completes the proof.
\end{proof}

\begin{corollary} \label{cor:StabA}
Working in the above setting, suppose that $\cC$ satisfies \prop{A}, $\omega$ satisfies \prop{Gen}, and \prop{Stab} holds. Then $\cD$ satisfies \prop{A}.
\end{corollary}

\begin{proof}
This follows from Proposition~\ref{prop:Agen}.
\end{proof}

\begin{remark}
Suppose that $\fg$ corresponds to a group $G$ and $\fh$ to a subgroup $H$. Intuitively, \prop{Gen} means that $G$ acts on some variety and there is a point with dense orbit that has stabilizer $H$.
\end{remark}

\subsection{Property (B)}

Let $\fg$ be a Lie algebra in $\cV$. Let $\cU=\cU(\fg)$. For $n \ge 0$, let $\cU_{\le n}=\cU_{\le n}(\fg)$ be the sum of the images of the maps $\fg^{\otimes k} \to \cU$ for $0 \le k \le n$.

\begin{proposition} \label{prop:Lie-zerodiv}
Let $A \in \Comm(\Rep(\fg))$. Let $X,Y \subset A$ be $\cV$-subobjects such that $XY=0$. Then we have $X^{n+1} \cdot \cU_{\le n} Y=0$ for any $n \ge 0$.
\end{proposition}

\begin{proof}
We first prove the proposition when $\cV$ is the category of abelian groups; thus $A$ is an ordinary commutative ring, $\fg$ is an ordinary Lie ring acting on $A$ by derivations, and $X$ and $Y$ are $\bZ$-submodules of $A$. We proceed by induction on $n$. The $n=0$ case is given. Suppose now that we have shown $X^{n+1} \cdot \cU_{\le n} Y=0$. Thus for $x \in X^{n+1}$, $y \in Y$, and $a \in \cU_{\le n}$, we have $x \cdot (ay)=0$. Let $E \in \fg$. Applying $E$ to this equation, we find $(Ex)(ay)+x \cdot(Eay)=0$. Since $x$ is a sum of $(n+1)$-fold products of elements of $X$ and $E$ acts by derivations, we see that $Ex \in X^n$. Thus if $w \in X$ is any element then $w \cdot Ex \in X^{n+1}$, and thus annihilates $ay$. Multiplying the previous equation by $w$, we therefore find $wx \cdot (Eay)=0$. Since this holds for all choices of $w$, $x$, $y$, $a$, and $E$, we find $X^{n+2} \cdot \cU_{\le n+1}Y=0$, as required.

We now treat the general case, using a functor of points approach. For an object $T$ of $\cV$, let $A(T)=\Hom_{\cV}(T,A)$ and $\fg(T)=\Hom_{\cV}(T,\fg)$. Then $A(T)$ is a commutative ring and $\fg(T)$ is a Lie ring acting on $A(T)$ by derivations. Let $T=X \oplus Y \oplus \fg$. Let $x \in A(T)$ be the map $T \to X \to A$, where the first map is the projection and the second is the inclusion, and define $y \in A(T)$ similarly. Then $xy=0$. By the previous paragraph, we have $x^{n+1} \cdot \cU_{\le n}(\fg(T)) y=0$. In particular, letting $a \in \fg(T)$ be the projection map $T \to \fg$, we find $x^{n+1} \cdot a^ky=0$ for $0 \le k \le n$. Regarding $x^{k+1} \cdot a^ky$ as a morphism $T \to A$, its image is $X^{n+1} \cdot \fg^k Y$, where $\fg^k$ is the image of $\fg^{\otimes k}$ in $\cU$. We thus see that $X^{n+1} \cdot \fg^k Y=0$ for $0 \le k \le n$, which completes the proof.
\end{proof}

Fix an admissible subcategory $\cC$ of $\Rep(\fg)$. We consider the following condition:
\begin{itemize}
\item[\prop{UF}\label{UF}] Let $M \in \cC$ be finitely generated. Then $M=\cU_{\le n} X$ for some $n$ and some $X \in [M]^{\rf}$.
\end{itemize}

\begin{proposition} \label{prop:UF-gen}
Suppose $\cC$ satisfies \prop{UF}. Let $M \in \cC$ be finitely generated, and let $X$ be a subobject generating $M$. Then $M=\cU_{\le n} X$ for some $n$.
\end{proposition}

\begin{proof}
By definition, we have $M=\cU_{\le m} Y$ for some $m$ and some finitely generated subobject $Y$ of $M$. Since $X$ generates $M$, we have $M=\cU X=\sum_{k \ge 0} \cU_{\le k} X$. Since $Y$ is contained in $\sum_{k \ge 0} \cU_{\le k} X$ and finitely generated, we have $Y \subset \cU_{\le k} X$ for some $k$. Thus $M=\cU_{\le m} \cU_{\le k} X \subset \cU_{\le m+k} X$. We can therefore take $n=m+k$.
\end{proof}

\begin{proposition}
If $\cC$ satisfies \prop{UF} then it satisfies \propn{B}{1}.
\end{proposition}

\begin{proof}
Let $A \in \Comm(\cC)$, and suppose $XY=0$ for $X,Y \in [A]^{\rf}$. By Proposition~\ref{prop:UF-gen}, we have $\cU Y =\cU_{\le n} Y$ for some $n$. Thus by Proposition~\ref{prop:Lie-zerodiv}, we have $X^{n+1} \cdot \cU Y=0$. It follows that $(\cU X^{n+1}) \cdot (\cU Y)=0$, which completes the proof.
\end{proof}

\begin{corollary} \label{cor:UFB}
If $\cC$ satisfies \prop{UF} and \prop{A} then it satisfies \prop{B}.
\end{corollary}

For the next two results, we fix a Lie subalgebra $\fh \subset \fg$ and an admissible subcategory $\cD \subset \Rep(\fh)$ that is compatible with $\cC$.

\begin{proposition} \label{prop:FinUF}
Suppose that $\cD$ satisfies \prop{UF} and the restriction functor $\cC \to \cD$ satisfies \prop{Fin}. Then $\cC$ also satisfies \prop{UF}.
\end{proposition}

\begin{proof}
Let $M \in \cC$ be finitely generated. Then $M$ is finitely generated as an $\fh$-representation, and so by \prop{UF} we have $M=\cU_{\le n}(\fh) X$ for some finitely generated subobject $X$ of $M$. Since $\cU_{\le n}(\fh) \subset \cU_{\le n}(\fg)$, we have $M=\cU_{\le n}(\fg) X$. Thus $\cC$ satisfies \prop{UF}.
\end{proof}

\begin{proposition} \label{prop:good-sub}
Suppose that $\cD$ satisfies \prop{UF} and \prop{A}, and the functor $\cC \to \cD$ satisfies \prop{Fin}. Then $\cC$ satisfies \prop{A}, \prop{B} and \prop{UF}.
\end{proposition}

\begin{proof}
Since $\cC \to \cD$ satisfies \prop{Fin} it also satisfies \prop{A} by Proposition~\ref{prop:FinA}; thus $\cC \to \cV$ satisfies \prop{A} by Proposition~\ref{prop:compA}. Proposition~\ref{prop:FinUF} shows that $\cC$ satisfies \prop{UF}, and so $\cC \to \cV$ satisfies \prop{B} by Corollary~\ref{cor:UFB}.
\end{proof}

\section{Applications to Lie superalgebras} \label{s:super}

\subsection{General remarks}

Fix a field $\bk$ of characteristic~0 throughout this section. Recall that a \defn{super vector space} over $\bk$ is a $\bZ/2$-graded vector space $V=V_0 \oplus V_1$. Given two super vector spaces $V$ and $W$, a linear map $f \colon V \to W$ is \defn{homogeneous} of degree $d \in \bZ/2$ if $f(V_i) \subset W_{i+d}$ for all $i \in \bZ/2$. We let $\SVec$ be the category whose objects are super vector spaces and whose morphisms are homogeneous morphisms of degree~0. Given a super vector space $V=V_0 \oplus V_1$, we let $V[1]$ be the super vector space given by $V[1]_i=V_{i+1}$; thus $(-)[1]$ switches the even and odd pieces. There is a natural isomorphism $V \to V[1]$ that is homogeneous of degree~1; note, however, that this does not count as an isomorphism in our category $\SVec$. Given two super vector spaces $V$ and $W$, their tensor product $V \otimes W$ is just their usual tensor product as graded vector spaces. We let $V \otimes W \to W \otimes V$ be the isomorphism defined by $x \otimes y \mapsto (-1)^{\vert x \vert \vert y \vert} y \otimes x$, where $\vert x \vert \in \bZ/2$ denotes the degree of the homogeneous element $x \in V$. This defines a symmetry on the tensor product. In this way, $\SVec$ is an admissible tensor category.

Let $\fg$ be a Lie superalgebra, that is, a Lie algebra object of $\SVec$. We let $\Rep(\fg)$ be the category of representations of $\fg$ on super vector spaces over $\bk$. As in the previous paragraph, the morphisms in this category are required to be homogeneous of degree~0. There is a natural forgetful functor $\Rep(\fg) \to \SVec$.

Let $\cC$ be an admissible subcategory of $\Rep(\fg)$. Suppose $A$ is a commutative algebra in $\cC$ and let $\fa$ be an ideal of $A$. We say that $\fa$ is \defn{$\fg$-prime} if it is a prime ideal in $\cC$. We write $\rad_{\fg}(\fa)$ for the radical of $\fa$ in $\cC$, and refer to this as the \defn{$\fg$-radical}; we let $\rad(\fa)$ be the usual radical of $\fa$ in the ring $A$. We let $\Spec_{\fg}(A)$ be the spectrum of $A$ as an algebra in $\cC$ (i.e., the set of $\fg$-primes), and we let $\Spec(A)$ denote the spectrum of the ordinary ring $A$.

\begin{proposition} \label{prop:simpleUF}
Suppose the following two conditions hold:
\begin{itemize}
\item Every finitely generated object of $\cC$ has finite length.
\item If $M \in \cC$ is simple then $M=\cU_{\le n} \cdot V$ for some $n$ and some finite dimensional subspace $V$ of $M$.
\end{itemize}
Then $\cC$ satisfies \prop{UF}.
\end{proposition}

\begin{proof}
Say that $M \in \cC$ is \defn{good} if $M=\cU_{\le n} \cdot V$ for some $n$ and some finite dimensional subspace $V$ of $M$. Thus all simple objects are good, and to show that $\cC$ satisfies \prop{UF} we must show that all finite length objects are good. It thus suffices to show that if
\begin{displaymath}
0 \to M_1 \to M_2 \to M_3 \to 0
\end{displaymath}
is a short exact sequence in $\cC$ with $M_1$ and $M_3$ good then $M_2$ is good. Thus suppose such a sequence is given. Write $M_1=\cU_{\le n} V$ and $M_3=\cU_{\le m} W$. Let $\wt{W}$ be a finite dimensional subspace of $M_2$ surjecting onto $W$. Given $x \in M_2$, we can find $y \in \cU_{\le m} W$ such that $x$ and $y$ have the same image in $M_3$. Thus $x-y \in M_1$, and so $x-y=z$ for some $z \in \cU_{\le n} V$. We thus find $M_2=\cU_{\le \max(n,m)} (V+\wt{W})$, which completes the proof.
\end{proof}

We now describe a method of constructing admissible subcategories. Let $S$ be a set of $\fg$-modules. Let $S_1$ be the set of all $\fg$-modules of the form $X_1 \otimes \cdots \otimes X_r$ with $X_1, \ldots, X_r \in S$. Let $S_2$ be the class of all $\fg$-modules of the form $\bigoplus_{i \in I} M_i$ with $M_i \in S_1$. Finally, let $S_3$ be the class of all $\fg$-modules that occur as a subquotient of a $\fg$-module in $S_2$. We define the tensor subcategory of $\Rep(\fg)$ \defn{generated} by $S$ to be the full subcategory of $\Rep(\fg)$ spanned by $S_3$. It is easily seen to be a Grothendieck abelian category and closed under tensor product. Furthermore, if every object in $S_1$ has finite length, then it is an admissible subcategory of $\Rep(\fg)$, as defined in Definition~\ref{def:admissible}.

\subsection{Polynomial representations of $\fgl$}

Let
\begin{displaymath}
\fgl=\fgl_{\infty|\infty}=\bigcup_{n \ge 1} \fgl_{n|n}, \qquad
\bV=\bC^{\infty|\infty}=\bigcup_{n \ge 1} \bC^{n|n}.
\end{displaymath}
Then $\bV$ is naturally a representation of $\fgl$, and we call it the \defn{standard representation}. In this section, we review the polynomial representation theory of $\fg$; we refer to \cite[\S 3.2]{chengwang} for more detailed information.

Let $\Rep^{\npol}(\fgl)$ be the tensor subcategory of $\Rep(\fg)$ generated by $\bV$. We refer to representations in this category as \defn{narrow polynomial representations}. (Note: ``narrow'' is not standard terminology.) Since the tensor powers of $\bV$ are finite length, this is an admissible subcategory. In fact, $\Rep^{\npol}(\fgl)$ is a semisimple abelian category, and every simple object has the form $\bS_{\lambda}(\bV)$ for a partition $\lambda$; here $\bS_{\lambda}$ denotes the Schur functor. It follows that $\Rep^{\npol}(\fgl)$ is equivalent (as a tensor category) to the classical category of polynomial representations of $\fgl_{\infty}$ on (non-super) vector spaces.

Let $\Rep^{\pol}(\fgl)$ be the tensor subcategory of $\Rep(\fgl)$ generated by $\bV$ and $\bV[1]$. We refer to representations in this category as \defn{wide polynomial representations}, but we often omit the word ``wide'' (which, again, is non-standard terminology.) This is an admissible subcategory, and semi-simple. The simple objects are now of the form $\bS_{\lambda}(\bV)$ or $\bS_{\lambda}(\bV)[1]$. Thus, as an abelian category, $\Rep^{\pol}(\fgl)$ is equivalent to a direct sum of two copies of $\Rep^{\npol}(\fgl)$.

\begin{proposition}
The category $\Rep^{\pol}(\fgl)$ satisfies \prop{UF}.
\end{proposition}

\begin{proof}
Let $\{e_i\}_{i \ge 1}$ be a basis for the even part of $\bV$, and let $v=e_1 \otimes \cdots \otimes e_n \in \bV^{\otimes n}$. One easily sees that $\bV^{\otimes n}= \cU_{\le n} v$. Thus, in the terminology of the proof of Proposition~\ref{prop:simpleUF}, the representation $\bV^{\otimes n}$ is good. Since $\bS_{\lambda}(\bV)$ is a quotient of $\bV^{\otimes n}$, with $n=\vert \lambda \vert$, we see that $\bS_{\lambda}(\bV)$ is good. It follows that $\bS_{\lambda}(\bV)[1]$ is also good. Thus \prop{UF} holds by Proposition~\ref{prop:simpleUF}.
\end{proof}

\begin{proposition}
The category $\Rep^{\pol}(\fgl)$ satisfies \prop{A}.
\end{proposition}

\begin{proof}
In \cite{tcaprimes}, it is shown that $\Rep^{\npol}(\fgl)$ satisfies \prop{A}. We deduce the present result from this. Let $A$ be a commutative algebra in $\Rep^{\pol}(\fgl)$. We then have a canonical decomposition $A=A_0 \oplus A_1$ in $\Rep^{\pol}(\fgl)$ where $A_0$ is a sum of $\bS_{\lambda}(\bV)$'s and $A_1$ is a sum of $\bS_{\lambda}(\bV)[1]$'s. Note that $A_0$ is \defn{not} the degree~0 piece of the super vector space $A$: the representation $\bS_{\lambda}(\bV)$ always has even and odd elements (for $\lambda$ non-empty). One easily sees that $A_0$ is a subalgebra of $A$, and $A_1$ is an $A_0$-submodule satisfying $A_1^2 \subset A_0$.

Let $X$ be a finite length subrepresentation of $A$ consisting of nilpotent elements. Write $X=X_0 \oplus X_1$ as above; of course, every element of $X_0$ or $X_1$ is nilpotent. We can regard $A_0$ as an object in $\Rep^{\npol}(\fgl)$. Thus, by \cite{tcaprimes}, we see that $X_0$ is nilpotent, i.e., $X_0^n=0$ for some $n$. Since $X_1^2$ is a subobject of $A_0$ consisting of nilpotent elements, it too is nilpotent; thus $X_1^{2m}=0$ for some $m$. We therefore find that $X^{n+2m}=0$, and so $X$ is nilpotent. This completes the proof.
\end{proof}

\begin{corollary}
The category $\Rep^{\pol}(\fgl)$ satisfies \prop{B}.
\end{corollary}

\begin{proof}
This follows from Corollary~\ref{cor:UFB}.
\end{proof}

Let $\fgl^n=\fgl \times \cdots \times \fgl$, where there are $n$ copies of $\fgl$. We define $\Rep^{\npol}(\fgl^n)$ and $\Rep^{\pol}(\fgl^n)$ in the obvious manner.

\begin{proposition}
The category $\Rep^{\pol}(\fgl^n)$ satisfies \prop{UF}, \prop{A} and \prop{B}.
\end{proposition}

\begin{proof}
Regard $\fgl$ as the diagonal subalgebra of $\fgl^n$. Then the restriction functor $\Rep^{\pol}(\fgl^n) \to \Rep^{\pol}(\fgl)$ satisfies \prop{Fin}; this follows from the fact that the tensor product of two finite length representations of $\fgl$ is again finite length. Thus the result follows from Proposition~\ref{prop:good-sub}.
\end{proof}

\subsection{Algebraic representations of $\fgl$}

Let $\bV_*=\bigcup_{n \ge 1} (\bC^{n|n})^*$ be the so-called \defn{restricted dual} of $\bV$. It is naturally a representation of $\fgl$. We define $\Rep^{\nalg}(\fgl)$, the category of \defn{narrow algebraic representations}, to be the tensor subcategory of $\Rep(\fgl)$ generated by $\bV$ and $\bV_*$. Similarly, we define $\Rep^{\alg}(\fgl)$, the category of \defn{(wide) algebraic representations}, to be the tensor subcategory of $\Rep(\fgl)$ generated by $\bV$, $\bV[1]$, $\bV_*$, and $\bV_*[1]$. The category $\Rep^{\nalg}(\fgl)$ is equivalent to the category of algebraic representations of $\fgl_{\infty}$ studied in \cite{koszulcategory, penkovserganova, penkovstyrkas, infrank}. It is \emph{not} semisimple. It is not difficult to show that if $V$ and $W$ are narrow algebraic representations then $\Hom_{\fgl}(V, W[1])=0$. Thus every wide algebraic representation canonically decomposes as $V \oplus W[1]$ with $V$ and $W$ narrow, and so $\Rep^{\alg}(\fgl)$ is equivalent to a direct sum of two copies of $\Rep^{\nalg}(\fgl)$.

Let $i \colon \fgl \to \fgl \times \fgl$ be the map $i(X)=(X, -X^t)$. Then $i$ is an injective homomorphism of Lie superalgebras; we refer to $i(\fgl)$ as the \defn{twisted diagonal subalgebra} of $\fgl \times \fgl$. Let $\bV_1$ and $\bV_2$ be copies of $\bV$ on which $\fgl \times \fgl$ act through the first and second projections. Then the restriction of $\bV_1$ via $i$ is $\bV$, while the restriction of $\bV_2$ is $\bV_*$. It follows that polynomial representations of $\fgl \times \fgl$ restrict to algebraic representations of $\fgl$ via $i$. In particular, we see that $\Rep^{\pol}(\fgl \times \fgl)$ and $\Rep^{\alg}(\fgl)$ are compatible admissible subcategories.

\begin{proposition}
Condition \prop{Stab} holds.
\end{proposition}

\begin{proof}
Let $\fb$ be the standard Borel subalgebra of $\fgl$. If $V$ is a polynonimal representation of $\fgl$, then the $\fb$-submodule generated by any element is finite dimensional. The same applies to $\fgl \times \fgl$ and $\fb \times \fb$. Since $\fgl \times \fgl = (\fb \times \fb) + i(\fgl)$, the result follows.
\end{proof}

\begin{proposition}
The restriction functor $\Rep^{\pol}(\fgl \times \fgl) \to \Rep^{\alg}(\fgl)$ satisfies \prop{Gen}.
\end{proposition}

\begin{proof}
Let $\{e_i,f_j\}$ be a basis for $\bV$, where the $e$ vectors are even and the $f$ vectors are odd. Let $R=\Sym(\bV_1 \otimes \bV_2)$, regarded as an algebra object in $\Rep^{\pol}(\fgl \times \fgl)$. Let
\begin{displaymath}
x_{i,j}=e_i \otimes e_j, \quad x'_{i,j}=f_i \otimes f_j, \quad y_{i,j} = e_i \otimes f_j, \quad y'_{i,j}=f_i \otimes e_j.
\end{displaymath}
Then $R$ is the super polynomial algebra in these variables; note that the $x$ and $x'$ variables are even, while the $y$ and $y'$ variables are odd. Let $\fm \subset R$ be the ideal generated by the elements
\begin{displaymath}
x_{i,j}-\delta_{i,j}, \quad x'_{i,j}-\delta_{i,j}, \quad y_{i,j}, \quad y'_{i,j}.
\end{displaymath}
Of course, $R/\fm \cong \bk$ and so (i) of \prop{Gen} holds. The ideal $\fm$ is stable under the twisted diagonal subalgebra $i(\fgl)$. It follows that if $M$ is a module object for $R$ in the category $\Rep^{\pol}(\fgl \times \fgl)$ then $M/\fm M$ is a $\fgl$-module, necessarily algebraic. In \cite[\S 3.5]{sym2noeth}, it is shown (ii) and (iii) of \prop{Gen} hold. (Actually, \cite{sym2noeth} only works with narrow polynomial representations of $\fgl_{\infty}$, but the same arguments apply in the present situation, essentially because $R$ itself is a narrow polynomial representation.)
\end{proof}

\begin{proposition} \label{prop: glUF}
The category $\Rep^{\alg}(\fgl)$ satisfies \prop{UF}.
\end{proposition}

\begin{proof}
We use the ``good'' terminology from the proof of Proposition~\ref{prop:simpleUF}. Let $\{e_i,f_j\}$ be as in the previous proof, and let $\{e_i^*,f_j^*\}$ be the dual basis of $\bV_*$. We let $[r]=\{1,\ldots,r\}$ for $r \le \infty$.

We claim that $V_{n,m}=\bV^{\otimes n} \otimes \bV^{\otimes m}$ is good. Let $r=n+m$ and for $T \in [\infty]^r$ put
\begin{displaymath}
e_T = e_{T_1} \otimes \cdots \otimes e_{T_n} \otimes e_{T_{n+1}}^* \cdots \otimes e_{T_{n+m}}^*.
\end{displaymath}
Since $V_{n,m}$ has finite length, there is some $s$ such that the $e_T$ with $T \in [s]^r$ generate it. Let $W$ be the span of these $e_T$'s. We thus have $V_{n,m}=\bigcup_{k \ge 1} \cU_{\le k} \cdot W$. We can thus find a single integer $k$ such that $e_T \in \cU_k \cdot W$ for any $T \in [r+s]^r$. Given any $T$, there exists a permutation $\sigma$ of $[\infty]$ fixing $1, \ldots, s$ such that $\sigma T \in [r+s]^r$. Since $e_{\sigma T} \in \cU_{\le k} \cdot W$ and $\cU_{\le k}$ and $W$ are stable by $\sigma$, we find $e_T \in \cU_{\le k} \cdot W$. Thus $V_{n,m}=\cU_{\le k} \cdot W$, and so the claim follows.

Since $V_{n,m}$ is good, so is $V_{n,m}[1]$. Any simple object is a quotient of some such representation, and therefore good. Thus \prop{UF} follows from Proposition~\ref{prop:simpleUF}.

(We now explain the claim about simple objects. Let $R$ be as in the previous proof, and let $T \colon \Mod_R \to \Mod_R^{\gen}$ be the quotient functor. Let $L$ be a simple object of $\Rep^{\alg}(\fgl)$, and let $T(M)$ be the corresponding simple object of $\Mod_R^{\gen}$, where $M$ is an $R$-module. Since $M$ is a polynomial representation of $\fgl \times \fgl$, we can find a surjection $R \otimes V \to M$, where $V$ is a sum of representations of the form $(\bV_1^{\otimes n} \otimes \bV_2^{\otimes m})[k]$. We thus have a surjection $T(R \otimes V) \to T(M)$. Since $T(M)$ is simple, it follows that some summand maps surjectively to it, that is, we have a surjection $T(R \otimes (\bV_1^{\otimes n} \otimes \bV_2^{\otimes m})[k]) \to T(M)$. Passing through the equivalence $\Mod_R^{\gen} \cong \Rep^{\alg}(\fgl)$, this yields a surjection $V_{n,m}[k] \to L$, as required.)
\end{proof}


\begin{corollary} \label{cor: glAB}
The category $\Rep^{\alg}(\fgl)$ satisfies \prop{A} and \prop{B}.
\end{corollary}

\begin{proof}
We obtain \prop{A} from Corollary~\ref{cor:StabA} and \prop{B} from Corollary~\ref{cor:UFB}.
\end{proof}

\begin{proposition} \label{prop:spec-noeth}
Let $A$ be a commutative algebra in $\Rep^{\alg}(\fgl)$. Suppose $A$ is generated over a noetherian coefficient ring by a finite length subrepresentation. Then $\Spec_{\fgl}(A)$ is a noetherian topological space.
\end{proposition}

\begin{proof}
Let $\fg \subset \fgl_{\infty|\infty}$ be the diagonal $\fgl_{\infty}$ inside of the even subalgebra of $\fgl$, and let $B=A/\rad(A)$. The ideal $\rad(A)$ is not $\fgl_{\infty|\infty}$-stable, but it is $\fg$ stable, and so $B$ is a $\fg$-algebra. One easily sees that $B$ belongs to $\Rep^{\alg}(\fg)$ and that it is equivariantly finitely generated. A variant of Draisma's theorem \cite{draisma} for algebraic representations, proved in \cite{ES}, implies that radical $\fg$-ideals of $B$ satisfy the ascending chain condition. Suppose now that $\fa_1 \subset \fa_2 \subset \cdots$ is an ascending chain of $\fgl$-radical ideals of $A$. Then $\rad(\fa_1) \subset \rad(\fa_2) \subset \cdots$ is (or, rather, corresponds to) an ascending chain of radical $\fg$-ideals of $B$, and thus stabilizes. Thus, by \prop{A}, the original chain stabilizes. This establishes the result.
\end{proof}

\subsection{The isomeric algebra} \label{s:isomericEx}

Let $\bC^{1|1}$ have basis $\epsilon_0, \epsilon_1$ and let $\beta \colon \bC^{1|1} \to \bC^{1|1}$ be the map defined by $\beta(\epsilon_i)=\epsilon_{i+1}$. Let $\bW=\bV \otimes \bC^{1|1}$, and let $\alpha \colon \bW \to \bW$ be the map $\id_{\bV} \otimes \beta$. The \defn{isomeric Lie superalgebra}\footnote{Commonly known as the ``queer lie superalgebra"} $\fq$ is the subalgebra of $\fgl(\bW)$ consisting of those elements that supercommute with $\alpha$. (Here $\fgl(\bW)$ means the copy of $\fgl_{\infty|\infty}$ associated to $\bW$.) We define the category $\Rep^{\pol}(\fq)$ of \defn{polynomial representations} to be the subcategory of $\Rep(\fq)$ generated by $\bW$. This category is semisimple, and somewhat similar to the category of polynomial representations of $\fgl$; see \cite[\S 3]{chengwang}. We define the category $\Rep^{\alg}(\fq)$ of \defn{algebraic representations} to be the subcategory generated by $\bW$ and $\bW_*$. This category was studied in \cite{grantcharov}.

If $X$ is an element of $\fgl(\bV)$ then $X \otimes 1$ is an endomorphism of $\bW$ that supercommutes with $\alpha$, and thus is an element of $\fq$. This defines an embedding of Lie superalgebras $\fgl=\fgl(\bV) \to \fq$. By putting an appropriate order on the basis of $\bW$, this embedding is given in terms of matrices by
\begin{displaymath}
\begin{pmatrix}
A &B\\
C &D
\end{pmatrix}
\mapsto
\begin{pmatrix}
A & 0 &0&B\\
0&D&C&0\\
0&B&A&0\\
C&0&0&D
\end{pmatrix}.
\end{displaymath}
Here the source matrix is decomposed into blocks according to the decomposition of $\bV$ into its even and odd pieces. We have:

\begin{proposition} \label{prop: isomericRes}
The restriction function $\Rep^{\alg}(\fq) \to \Rep^{\alg}(\gl)$ satisfies \prop{Fin}.
\end{proposition}

\begin{proof}
It suffices to check that the generators of $\Rep^{\alg}(\fq)$ are finitely generated as $\fgl$-representations. The representation $\bW$ of $\fq$ restricts to the representation $\bV \oplus \bV[1]$, while $\bW_*$ restricts to $\bV_* \oplus \bV_*[1]$. Thus the result follows.
\end{proof}

\begin{theorem} \label{thm: isomericABUF}
The category $\Rep^{{\rm alg}}(\fq)$ satisfies \prop{A}, \prop{B} and \prop{UF}.
\end{theorem}

\begin{proof}
This follows from applying Proposition \ref{prop:good-sub} to the restriction functor $\Rep^{{\rm alg}}(\fq) \to \Rep^{{\rm alg}}(\gl)$. The necessary conditions are satisfied via Propositions \ref{prop: isomericRes}, \ref{prop: glUF} and Corollary \ref{cor: glAB}.
\end{proof}

\begin{proposition} \label{prop: nisomericABUF}
The category $\Rep^{\rm alg}(\fq^n)$ satisfies \prop{A}, \prop{B}, \prop{UF}.
\end{proposition}

\begin{proof}
Regard $\fq$ as the diagonal subalgebra of $\fq^n$. Then the restriction functor $\Rep^{\rm alg}(\fq^n) \to \Rep^{\rm alg}(\fq)$ satisfies \prop{Fin}; this follows from the fact that the tensor product of two finitely generated representations of $\fq$ is again finitely generated. Thus, the result follows from Theorem \ref{thm: isomericABUF} in combination with Proposition \ref{prop:good-sub}.
\end{proof}

\begin{proposition} \label{prop:qspec-noeth}
Let $A$ be a commutative algebra in $\Rep^{\alg}(\fq)$. Suppose $A$ is generated over a noetherian coefficient ring by a finite length subrepresentation. Then $\Spec_{\fq}(A)$ is a noetherian topological space.
\end{proposition}

\begin{proof}
Let $\fa_1 \subset \fa_2 \subset \cdots$ be an ascending chain of $\fq$-radical ideals of $A$. Since $A$ is finitely generated as a $\fgl$-algebra, it follows from Proposition~\ref{prop:spec-noeth} that the chain $\rad_{\fgl}(\fa_{\bullet})$ stabilizes. Thus the chain $\rad(\fa_{\bullet})=\rad(\rad_{\gl}(\fa_{\bullet}))$ stabilizes, and so, by \prop{A}, the chain $\fa_{\bullet}$ stabilizes. (Note: we are essentially just applying \prop{A} for the functor $\Rep^{\alg}(\fq) \to \Rep^{\alg}(\fgl)$ here.) The result follows.
\end{proof}

\begin{remark}
The more obvious restriction functor $\Rep^{\alg}(\fq) \to \Rep^{\alg}(\gl_\infty)$, where $\fgl_{\infty}$ is the even subalgebra of $\fq$, is not sufficient for our purposes because $\Rep^{\alg}(\gl_\infty)$ does not satisfy \prop{A}.
\end{remark}

\subsection{The orthosymplectic algebra}

Let $\bW=\bV \oplus \bV_*$. This space carries a canonical even symmetric bilinear form. The \defn{orthosymplectic Lie superalgebra} $\osp$ is the stabilizer of this form inside of $\fgl(\bW)$. We define the category $\Rep^{\alg}(\osp)$ of \defn{(wide) algebraic representations} of $\osp$ to be the subcategory of $\Rep(\osp)$ generated by $\bW$. As in the $\fgl$ case, there is also a narrow category, which is equivalent to the category of algebraic representations of the infinite orthogonal category; this category was studied in \cite{koszulcategory, penkovserganova, penkovstyrkas,infrank}.

Any element of $\fgl=\fgl(\bV)$ acts on $\bW=\bV \oplus \bV_*$ and preserves the form. This induces an embedding $\fgl \to \osp$ of Lie superalgebras. In terms of matrices, this embedding is given (in a suitable basis) by
\begin{displaymath}
\begin{pmatrix}
A &B\\
C &D
\end{pmatrix}
\mapsto
\begin{pmatrix}
A & 0 & B & 0\\
0 & -A^t & 0 & -C^t\\
C & 0 & D & 0\\
0 & B^t & 0 &-D^t\\
\end{pmatrix},
\end{displaymath}
As before, we have:

\begin{proposition} \label{prop: ospRes}
The restriction functor $\Rep^{\alg}(\osp) \to \Rep^{\alg}(\gl)$ satisfies (Fin).
\end{proposition}

\begin{proof}
Since the generator $\bW$ of $\Rep^{\alg}(\osp)$ restricts to $\bV \oplus \bV_*$, which is a finite length representation, the result follows.
\end{proof}

\begin{theorem}
The category $\Rep^{{\rm alg}}(\osp)$ satisfies (A), (B) and (UF).
\end{theorem}

\begin{proof}
This follows from applying Proposition \ref{prop:good-sub} to the restriction functor $\Rep^{{\rm alg}}(\osp) \to \Rep^{{\rm alg}}(\gl)$. The necessary conditions are satisfied via Propositions \ref{prop: ospRes}, \ref{prop: glUF} and Corollary \ref{cor: glAB}.
\end{proof}

\begin{proposition}
Let $A$ be a commutative algebra in $\Rep^{\alg}(\osp)$. Suppose $A$ is generated over a noetherian coefficient ring by a finite length subrepresentation. Then $\Spec_{\osp}(A)$ is a noetherian topological space.
\end{proposition}

\begin{proof}
The proof is identical to that of Proposition~\ref{prop:qspec-noeth}
\end{proof}

\subsection{The periplectic algebra}

Let $\bW=\bV \oplus \bV_*[1]$. This space carries a canonical odd symmetric bilinear form. The \defn{periplectic Lie superalgebra} $\pe$ is the stabilizer of this form inside of $\fgl(\bW)$. We define the category $\Rep^{\alg}(\pe)$ of \defn{algebraic representations} of $\pe$ to be the subcategory of $\Rep(\pe)$ generated by $\bW$. This category was studied in \cite{serganova}.

Every element of $\fgl=\fgl(\bV)$ induces an map of $\bW$ that is compatible with the pairing, and so there is an embedding $\fgl \to \pe$. In terms of matrices, it is given by
\begin{displaymath}
\begin{pmatrix}
A &B\\
C &D
\end{pmatrix}
\mapsto
\begin{pmatrix}
A & 0 & 0 & B\\
0 & -D^t & B^t & 0\\
0 & -C^t & -A^t & 0\\
C & 0 & 0 & D
\end{pmatrix}.
\end{displaymath}
As in the other cases, we have:

\begin{proposition} \label{prop: peRes}
The restriction function $\Rep^{\alg}(\pe) \to \Rep^{\alg}(\gl)$ satisfies (Fin).
\end{proposition}

\begin{proof}
The generator $\bW$ of $\Rep^{\alg}(\pe)$ restricts to the finite length representation $\bV \oplus \bV_*[1]$, and so the result follows.
\end{proof}

\begin{theorem}
The category $\Rep^{{\rm alg}}(\pe)$ satisfies (A), (B) and (UF).
\end{theorem}

\begin{proof}
This follows from applying Proposition \ref{prop:good-sub} to the restriction functor $\Rep^{{\rm alg}}(\pe) \to \Rep^{{\rm alg}}(\gl)$. The necessary conditions are satisfied via Propositions \ref{prop: peRes}, \ref{prop: glUF} and Corollary \ref{cor: glAB}.
\end{proof}

\begin{proposition}
Let $A$ be a commutative algebra in $\Rep^{\alg}(\pe)$. Suppose $A$ is generated over a noetherian coefficient ring by a finite length subrepresentation. Then $\Spec_{\pe}(A)$ is a noetherian topological space.
\end{proposition}

\begin{proof}
The proof is identical to that of Proposition~\ref{prop:qspec-noeth}
\end{proof}

\subsection{Additional comments}

We deduced \prop{A} and \prop{B} for $\Rep^{\alg}(\fgl)$ from corresponding properties for $\Rep^{\pol}(\fgl)$ via the criterion in \S \ref{ss:propA}. We then deduced the properties for $\fq$, $\osp$, and $\pe$ from the $\Rep^{\alg}(\fgl)$ case. Instead, one can also establish the properties for $\fq$, $\osp$, and $\pe$ in a parallel fashion to the $\Rep^{\alg}(\fgl)$ case. (In the $\fq$ case, this requires first establishing the properties for $\Rep^{\pol}(\fq)$; this follows easily from properties of the restriction functor $\Rep^{\pol}(\fq) \to \Rep^{\pol}(\fgl)$.) The key input required to carry this out is summarized in Figure~\ref{fig:cases}.

\begin{figure}[!h]
\begin{center}
\begin{tabular}{rrrrrrr}
\thickhline\\[-11pt]
&\;& $\fg$ &\;& $\fh$ &\;& $E$ \\[2pt]
\hline\\[-10pt]
I. && $\fgl \times \fgl$ && $\fgl$ && $\bV \otimes \bV$ \\[2pt]
II. && $\fgl$ && $\osp$ && $\Sym^2(\bV)$ \\[2pt]
III. && $\fgl$ && $\pe$ && $\Sym^2(\bV)[1]$ \\[2pt]
IV. && $\fq \times \fq$ && $\fq$ && $2^{-1}(\bV \otimes \bV)$ \\[3pt]
\thickhline
\end{tabular}
\end{center}
\label{fig:cases}
\caption{In each case, there is a restriction functor $\Rep^{\pol}(\fg) \to \Rep^{\alg}(\fh)$ that satisfies \prop{Gen}. The algebra $R$ in $\Rep^{\pol}(\fg)$ is $\Sym(E)$. Cases~I and~II are established in \cite{sym2noeth}; case~III in \cite{periplectic}; and case~IV in \cite{isomeric}.}
\end{figure}

\section{An example} \label{s:isomeric}

\subsection{Background}

In this section, we apply our theory to classify the equivariant primes in the isomeric algebra $A$ studied in \cite{isomeric}. We begin by briefly recalling some background material; we refer to \cite[\S 2]{isomeric} for a more detailed discussion and to \cite[\S 3]{chengwang} for general background on the isomeric algebra and its representations.

Recall that a \defn{isomeric vector space} is a pair $(V, \alpha)$ where $V$ is a super vector space and $\alpha$ is an odd degree automorphism of $V$ squaring to the identity (an isomeric structure). Given such a space, the \defn{isomeric Lie superalgebra} $\fq(V)$ is the subalgebra of $\fgl(V)$ consisting of endomorphisms $X$ that are compatible with $\alpha$ (i.e., $X\alpha=(-1)^{\vert X \vert} \alpha X$ for $X$ homogeneous). We say that a representation of $\fq(V)$ is \defn{polynomial} if it occurs as a subquotient in a direct sum of tensor powers of $V$. We let $\Rep^{\pol}(\fq(V))$ be the category of polynomial representations. It is a semisimple abelian category (even if $\dim(V)=\infty$).

The simple polynomial representations can be constructed uniformly, as follows. Consider the tensor power $V^{\otimes n}$. The symmetric group $\fS_n$ acts by permuting the tensor factors. Furthermore, for each $1 \le i \le n$, we can consider the endomorphism $\alpha_i$ induced by $\alpha$ on the $i$th tensor factor; note that for $i \ne j$, the endomorphisms $\alpha_i$ and $\alpha_j$ supercommute. The \defn{Clifford algebra} $\Cl_n$ is the superalgebra generated by $n$ supercommuting odd elements that square to~1, and the \defn{Hecke--Clifford algebra} $\cH_n$ is the semi-direct product algebra $\fS_n \ltimes \Cl_n$. The $\fS_n$ action and $\alpha_i$'s described above endow $V^{\otimes n}$ with the structure of an $\cH_n$-module. The simple $\cH_n$-modules are indexed by strict partitions of $n$ (i.e., partitions with no repeated parts). Given a simple $\cH_n$-module $L_{\lambda}$, we let $\bT_{\lambda}(V)=V^{\otimes n} \otimes_{\cH_n} L_{\lambda}$. If $\dim(V)<\ell(\lambda)$ then this space is~0. If $\dim(V) \ge \ell(\lambda)$ then $\bT_{\lambda}(V)$ is an irreducible polynomial representation of $\fq(V)$. Moreover, any irreducible polynomial representation is isomorphic to $\bT_{\lambda}(V)$ for a unique $\lambda$ with $\ell(\lambda) \le \dim(V)$. The construction $\bT_{\lambda}(V)$ is functorial in $V$ (with respect to maps of isomeric vector spaces), and $\bT_{\lambda}$ can be seen as a isomeric analog of a Schur functor.

If $V$ is an irreducible finite dimensional representation of a Lie superalgebra $\fg$ then $\End_{\fg}(V)$ is either one-dimensional (``type M'') or two-dimensional (``type Q''); in the latter case, the endomorphism ring is generated by a isomeric structure. The irreducible $\bT_{\lambda}(V)$ is type~M if $\ell(\lambda)$ is even and type~Q if $\ell(\lambda)$ is odd; in fact, this continues to hold if $V$ is infinite dimensional.

If $(V,\alpha)$ and $(W,\beta)$ are two isomeric vector spaces then $\alpha \otimes \beta$ is an even automorphism of $V \otimes W$ squaring to $-1$. The \defn{half tensor product} of $V$ and $W$, denoted $2^{-1}(V \otimes W)$, is the $\zeta_4$-eigenspace of $\alpha \otimes \beta$ on $V \otimes W$, where $\zeta_4$ is a fixed square root of $-1$. If $\fg$ and $\fh$ are Lie superalgebras and $V$ and $W$ finite dimensional irreducible representations then $V \otimes W$ is an irreducible representation of $\fg \times \fh$ if at least one of $V$ or $W$ has type~M; if $V$ and $W$ both have type~Q then $2^{-1}(V \otimes W)$ is an irreducible representation of $\fg \times \fh$ of type~M (see \cite[\S 3.1.3]{chengwang}). This also holds for polynomial representations of $\fq(V)$ in the infinite dimensional case.

\subsection{The ring $A$}

Let $(V,\alpha)$ and $(W,\beta)$ be isomeric vector spaces, and let $U$ be their half tensor product. We let $A=\Sym(U)$, which we regard as an algebra object in $\Rep^{\pol}(\fq(V) \times \fq(W))$. We are most interested in the case where $V$ and $W$ are infinite dimensional, though we also make use of the finite dimensional case. At times we treat $A$ as a functor of $(V,W)$. We have an analog of the Cauchy decomposition for $A$:
\begin{displaymath}
A = \bigoplus_{\lambda} 2^{-\delta(\lambda)} (\bT_{\lambda}(V) \otimes \bT_{\lambda}(W)),
\end{displaymath}
where the sum is over all strict partitions $\lambda$. We let $A_{\lambda}$ be the $\lambda$ summand in the above expression. This is an irreducible $\fq(V) \times \fq(W)$ representation, and if $\lambda \ne \mu$ then $A_{\lambda}$ and $A_{\mu}$ are non-isomorphic. Thus $A=\bigoplus A_{\lambda}$ is multiplicity free. As explained in \cite[\S 2.8]{isomeric}, $A$ is a polynomial superalgebra in even variables $x_{i,j}$ and odd variables $y_{i,j}$, with $i,j \ge 1$.

We let $I_{\lambda}$ be the ideal of $A$ generated by $A_{\lambda}$. By \cite[Theorem~1.2]{isomeric}, we have
\begin{displaymath}
I_{\lambda}=\bigoplus_{\lambda \subset \mu} A_{\mu}.
\end{displaymath}
In particular, we see that $I_{\lambda} \subset I_{\mu}$ if and only if $\mu \subset \lambda$.

We let $\Spec(A)$ be the spectrum of the ring $A$. This space has the structure of a superscheme. If $T$ is a superalgebra then a $T$-point of $\Spec(A)$ is a $T$-linear map
\begin{displaymath}
\langle , \rangle \colon V_T \otimes_T W_T \to T
\end{displaymath}
(where $V_T=T \otimes V$) of degree~0 that is compatible with the isomeric structures, in the sense that
\begin{displaymath}
\langle \alpha(v), \beta(w) \rangle = (-1)^{\vert v \vert} \zeta_4 \langle v, w \rangle,
\end{displaymath}
where $\vert v \vert \in \bZ/2$ denotes the degree of the homogeneous element $v \in V$. (We refer to such a map as a \defn{isomeric pairing}.) As a topological space, $\Spec(A)$ coincides with the spectrum of the reduced ring $A_{\rm red}$, which is an ordinary commutative ring. From the above, we see that a $\bC$-point of $\Spec(A)$ is a isomeric pairing $V \times W \to \bC$. Since the pairing is even, $V_i$ and $W_{i+1}$ must pair to~0. Furthermore, compatibility with the isomeric structures implies that the pairing is determined by its restriction to $V_0 \times W_0$. We thus see that $\Spec(A)$ is identified with $\Spec(\Sym(V_0 \otimes W_0))$ as a topological space. The isomeric supergroup $\bQ(V) \times \bQ(W)$ does not act on the latter space, but its even subgroup $\GL(V_0) \times \GL(W_0)$ does, and the identification is compatible with this action. We let $\Spec(A)_{\le r}$ be the locus consisting of points of rank $\le r$ (meaning the pairing $V_0 \times W_0 \to \bC$ has rank $\le r$). We also put $\Spec(A)_{\le \infty}=\Spec(A)$.

\subsection{Isomeric determinantal ideals}

Let $\sigma(r)$ be the ``staircase partition'' with $r$ rows, i.e., $(r, r-1, \ldots, 1)$, and let $I_r=I_{\sigma(r+1)}$. (We also put $I_{\infty}=0$.) We refer to $I_r$ as the \defn{isomeric determinantal ideal} of rank $r$. From the general decomposition of $I_{\lambda}$, we find
\begin{displaymath}
A/I_r = \bigoplus_{\ell(\lambda) \le r} A_{\lambda}.
\end{displaymath}
This decomposition will be important to what follows.

Fix $r<\infty$. We now study the ideal $I_r$ in more detail. Let $E$ be a isomeric space of dimension $r|r$. Put
\begin{displaymath}
B = \Sym(2^{-1}(V \otimes E) \oplus 2^{-1}(W \otimes E^*)).
\end{displaymath}
This is a superalgebra on which $\fq(V) \times \fq(W) \times \fq(E)$ acts.

\begin{proposition}
The invariant space $B^{\fq(E)}$ is (not necessarily naturally) isomorphic to $A/I_r$ as a representation of $\fq(V) \times \fq(W)$.
\end{proposition}

\begin{proof}
Appealing to the isomeric analog of the Cauchy decomposition, we have
\begin{align*}
B &= \Sym(2^{-1}(V \otimes E)) \otimes \Sym(2^{-1}(W \otimes E^*)) \\
&= \bigoplus_{\lambda,\mu} 2^{-\ell(\lambda)}(\bT_{\lambda}(V) \otimes \bT_{\lambda}(E)) \otimes 2^{-\ell(\mu)}(\bT_{\mu}(W) \otimes \bT_{\mu}(E^*)).
\end{align*}
The sum is taken over all strict partitions $\lambda$ and $\mu$ with $\le r$ rows. For $\lambda \ne \mu$, the irreducible representations $\bT_{\lambda}(E)$ and $\bT_{\mu}(E)$ are non-isomorphic, and so the $\fq(E)$-invariant space of $\bT_{\lambda}(E) \otimes \bT_{\mu}(E^*)$ vanishes. Since the half tensor product is associative and commutative up to isomorphism (see \cite[\S 7]{supermon}), we have
\begin{displaymath}
B^{\fq(E)} \cong \bigoplus_{\ell(\lambda) \le r} 2^{-\ell(\lambda)}(\bT_{\lambda}(V) \otimes \bT_{\lambda}(W)) \otimes 2^{-\ell(\lambda)}(\bT_{\lambda}(E) \otimes \bT_{\lambda}(E^*))^{\fq(E)}.
\end{displaymath}
The invariant space above is one-dimensional, and so the result follows.
\end{proof}

In particular, we see that the irreducible $U=2^{-1}(V \otimes W)$ appears with multiplicity one in $B^{\fq(E)}$ (assuming $r>0$). Since $U$ has type~M, there is a unique map of $(\fq(V) \times \fq(W))$-representations $U \to B^{\fq(E)}$, up to scaling. Thus, up to this ambiguity, there is a unique equivariant ring homomorphism $A \to B^{\fq(E)}$. Since all representations in $B{^\fq(E)}$ have $\le r$ rows, this map factors through $A/I_r$, and induces a homomorphism $A/I_r \to B^{\fq(E)}$. The above proposition suggests this map might be an isomorphism, which is confirmed by the following proposition.

\begin{proposition} \label{prop:qdet}
We have the following:
\begin{enumerate}
\item $V(I_r) \subset \Spec(A)$ is the rank $\le r$ locus $\Spec(A)_{\le r}$.
\item We have a natural isomorphism $A/I_r \to B^{\fq(E)}$ of $(\fq(V) \times \fq(W))$-algebras.
\end{enumerate}
\end{proposition}

\begin{proof}
First suppose that $V$ and $W$ be finite dimensional. Let $\QGr_r(V)$ be the isomeric Grassmannian, parametrizing isomeric quotients of $V$ of dimension $r|r$, and let $\cQ$ be the tautological bundle on it. Let $C=\Sym(\cQ \otimes W)$, regarded as a quasi-coherent sheaf of algebras on $\QGr_r(V)$, and let $\Isom(\cQ,E)$ denote the space of isomorphisms of the isomeric vector bundle $\cQ$ with the trivial isomeric bundle on $E$. Consider the diagram
\begin{displaymath}
\xymatrix@C=4em{
\QGr_r(V) \times \Spec(C) \ar[d]_{\alpha} & \QGr_r(V) \times \Spec(B) \times \Isom(\cQ, E) \ar[l]_-{\gamma} \ar[d]^{\delta} \\
\Spec(A) & \Spec(B) \ar[l]^{\beta} }
\end{displaymath}
To describe the maps, let $T$ be a superalgebra. A $T$-point of $\Gr_r(V) \times \Spec(C)$ consists of a isomeric quotient $V_T \to T^{r|r}$ and a isomeric pairing $T^{r|r} \times W_T \to T$. Composing, we get a isomeric pairing $V_T \times W_T \to T$, which is a $T$-point of $\Spec(A)$. This is the map $\alpha$. A $T$-point of $\Spec(B)$ consists of a isomeric map $W_T \to E_T$ and a isomeric pairing $V_T \times E_T \to T$. Again, the composition gives a isomeric pairing $V_T \times W_T \to T$, which is a $T$-point of $A$, and this is the map $\beta$. The map $\gamma$ is simply the projection map. Finally, a $T$-point of the top right space consists of a isomeric quotient $V_T \to T^{r|r}$, a isomeric pairing $T^{r|r} \times W \to T$, and a isomeric isomorphism $T^{r|r} \to E_T$. Out of this data we can naturally build a $T$-point of $\Spec(B)$, and this is the map $\delta$. It follows from the descriptions of the maps on $T$-points that the diagram commutes.

The map $\beta$ is $\fq(E)$-invariant, and thus coresponds to a ring homomorphism $\beta^* \colon A \to B^{\fq(E)}$. Since every representation of $\fq(V) \times \fq(W)$ appearing in $B^{\fq(E)}$ has $\le r$ rows, it follows that $\beta^*$ factors through $A/I_r$. Thus $\beta \circ \delta=\alpha \circ \gamma$ factors (scheme-theoretically) through $\Spec(A/I_r)$. Since $\gamma$ is simply a projection map, it follows that $\alpha$ similarly factors.

We claim that the map $\alpha^*$ on global functions is injective modulo $I_r$. Since $\alpha$ is equivariant for the actions of the isomeric groups, the kernel of $\alpha^*$ is a $(\fq(V) \times \fq(W))$-ideal of $A/I_r$. Since every representation appearing in $A/I_r$ has $\le r$ rows, it suffices to prove injectivity in the case where $V=W=\bC^{r|r}$. But in this case $\QGr_r(V)$ is a point, $C=A$, and $\alpha$ is the identity map. Thus the claim follows.

It follows from the previous paragraph that $\im(\alpha)$ is Zariski dense in $V(I_r)$. Since the image of $\alpha$ on $\bC$-points is the points of rank $\le r$, claim (a) follows.

The map $\gamma^*$ on functions is clearly injective. Thus $\delta^* \circ \beta^*=\gamma^* \circ \alpha^*$ is injective modulo $I_r$. In particular, we see that $\beta^* \colon A/I_r \to B^{\fq(E)}$ is injective. Since $A/I_r$ and $B^{\fq(E)}$ are isomorphic as representations, and all multiplicity spaces are finite dimensional, this map is necessarily an isomorphism. Thus (b) follows.

The infinite dimensional case follows from the finite dimensional case. Indeed, (a) can be rephrased as saying that $\rad(I_r)$ coincides with some $(\GL(V_0) \times \GL(W_0))$-ideal, and since the algebra is a polynomial representation of $\GL(V_0) \times \GL(W_0)$ we can check such an equality after evaluating on finite dimensional spaces. Similarly, the map in (b) can be checked to be an isomorphism after evaluating on finite dimensional spaces.
\end{proof}

\begin{lemma} \label{lem:q-integral}
Suppose $V$ and $W$ are infinite dimensional. Then the algebra $\Sym(V^{\oplus r} \oplus W^{\oplus s})$ is $(\fq(V) \times \fq(W))$-integral.
\end{lemma}

\begin{proof}
Let $f$ and $g$ be non-zero elements of the algebra. By picking bases, we can identify this algebra with a super polynomial ring. All zero divisors in this ring stem from odd degree variables squaring to~0. Since $V$ and $W$ are infinite dimensional, the representation generated by $f$ will contain a non-zero element $f'$ that has no variables in common with $g$. Thus $f' g \ne 0$, which establishes the result.
\end{proof}

\begin{remark} \label{rmk:Aintegral}
A similar argument shows that $A$ itself is $(\fq(V) \times \fq(W))$-integral.
\end{remark}

\begin{proposition} \label{prop:qdet-prime}
The ideal $I_r$ is $(\fq(V) \times \fq(W))$-prime.
\end{proposition}

\begin{proof}
By Lemma~\ref{lem:q-integral} the algebra $B$ is $(\fq(V) \times \fq(W))$-integral; indeed, note that $B$ is a subalgebra of $\Sym(V \otimes E \oplus W \otimes E^*)$, which has the form considered in the lemma. It follows that the subalgebra $B^{\fq(E)}$ is also $(\fq(V) \times \fq(W))$-integral. By Proposition~\ref{prop:qdet}, we see that $A/I_r$ is $(\fq(V) \times \fq(W))$-integral, and so $I_r$ is $(\fq(V) \times \fq(W))$-prime.
\end{proof}

\subsection{Classification of primes}

We now come to our main result. We suppose that $V$ and $W$ are infinite dimensional, and put $\fg=\fq(V) \times \fq(W)$ for brevity.

\begin{theorem}
We have the following:
\begin{enumerate}
\item The ideal $I_r$ is $\fg$-prime for all $0 \le r \le \infty$.
\item Every $\fg$-prime of $A$ is one of the $I_r$, for $0 \le r \le \infty$.
\item Every $\fg$-radical ideal of $A$ is $\fg$-prime, and thus also one of the $I_r$.
\end{enumerate}
\end{theorem}

\begin{proof}
(a) We have already proved this (Proposition~\ref{prop:qdet-prime} for $r<\infty$, and Remark~\ref{rmk:Aintegral} for $r=\infty$).

(b) Suppose now that $\fp$ is some $\fg$-prime of $A$. Since $\fp$ is stable by the group $\GL(V_0) \times \GL(W_0)$, so is the locus $V(\fp)$. It follows from basic linear algebra that the only $\GL(V_0) \times \GL(W_0)$ stable closed subsets of $\Spec(A)$ are the rank loci $\Spec(A)_{\le r}$. We thus see that $V(\fp)=V(I_r)$ for some $r$. Since $\Rep^{\pol}(\fg) \to \SVec$ satisfies \prop{A} (Proposition~\ref{prop: nisomericABUF}, note that \prop{A} passes to subcategories), we see that $\rad_{\fg}(\fp)=\rad_{\fg}(I_r)$. But since $\fp$ and $I_r$ are $\fg$-prime, they are equal to their $\fg$-radicals, and so $\fp=I_r$.

(c) A $\fg$-radical ideal is an intersection of $\fg$-prime ideals. Since the $\fg$-primes are totally ordered under inclusion by (b), the claim follows.
\end{proof}

\end{document}